\documentclass[11pt,reqno]{amsart}
\usepackage[a4paper,margin=27mm]{geometry}
\usepackage[T1]{fontenc}
\usepackage{lmodern,microtype,amsmath,amssymb,mathtools,mathrsfs,enumitem}
\usepackage{xcolor}
\usepackage[colorlinks=true,linkcolor=blue!50!black,citecolor=blue!50!black,urlcolor=blue!50!black,pagebackref=true]{hyperref}
\renewcommand*{\backref}[1]{}
\renewcommand*{\backrefalt}[4]{}
\hypersetup{pdftitle={Finite-Time Singularities of the Kahler--Ricci Flow on Fano Bundles II},
pdfauthor={Wangjian Jian and Jian Song},
pdfsubject={Finite-time singularities in the analytic minimal model program}}
\newtheorem{theorem}{Theorem}[section]
\newtheorem{conjecture}{Conjecture}[section]
\newtheorem{lemma}{Lemma}[section]
\newtheorem{proposition}[lemma]{Proposition}
\newtheorem{corollary}{Corollary}[section]
\theoremstyle{definition}
\theoremstyle{remark}\newtheorem{remark}[lemma]{Remark}
\numberwithin{equation}{section}
\newcommand{\dd}{\mathrm d}
\DeclareMathOperator{\tr}{tr}

\renewcommand{\P}{\mathbb P}
\newcommand{\C}{\mathbb C}\newcommand{\R}{\mathbb R}\newcommand{\Q}{\mathbb Q}
\newcommand{\PP}{\mathbb P}\newcommand{\ddc}{\sqrt{-1}\partial\bar\partial}
\newcommand{\Ric}{\operatorname{Ric}}\newcommand{\Rm}{\operatorname{Rm}}
\newcommand{\Reg}{\operatorname{Reg}}
\newcommand{\Vol}{\operatorname{Vol}}\newcommand{\diam}{\operatorname{diam}}
\newcommand{\Var}{\operatorname{Var}}\newcommand{\Ent}{\operatorname{Ent}}
\newcommand{\supp}{\operatorname{supp}}
\newcommand{\Nash}{\mathcal N}\newcommand{\Ding}{\mathcal D}
\newcommand{\dV}{\,dV}\newcommand{\eps}{\varepsilon}
\newcommand{\Fconv}{\mathbb F}
\setlist[enumerate]{label=\textup{(\roman*)},leftmargin=2.2em}
\title[Finite-time singularities on Fano bundles II]{Finite-Time Singularities of the K\"ahler--Ricci Flow on Fano Bundles II}
\author{Wangjian Jian}
\address{State Key Laboratory of Mathematical Sciences and Institute of Mathematics, Academy of Mathematics and Systems Science,
Chinese Academy of Sciences, Beijing, 100190, China}
\email{wangjian@amss.ac.cn}
\thanks{Wangjian Jian is supported in part by NSFC grants 12422103, 12201610,
12371058, and 12288201; BJNSF grant JR25002; and National Key R\&D Program
of China grants 2023YFA1009900 and 2021YFA1003100.}
\author{Jian Song}
\address{Department of Mathematics, Rutgers University,
Piscataway, NJ 08854, USA}
\email{jiansong@math.rutgers.edu}
\thanks{Jian Song is supported in part by the National Science Foundation
grant DMS-2505575.}
\date{}
\begin{document}
\begin{abstract}
 The analytic minimal model program proposed in \cite{SongTianSingularities, JST} seeks to describe the
birational transitions and fibre collapsing of the K\"ahler--Ricci
flow through the geometry of its singularities. A fundamental conjecture of  \cite{JST} predicts Type I bounds for the scalar
curvature and diameter of every fibre contratced by the limiting cohomological class at the finite singular time.  This paper is a continuation of \cite{Splitting} that establishes the Type I conjecture for collapsing solutions on Fano bundles with arbitrary smooth Fano fibres. If the fibre
admits a smooth K\"ahler-Einstein metric, the Type I bound holds 
for the full curvature tensor. 

\end{abstract}
\maketitle
\setcounter{tocdepth}{1}
\tableofcontents
\section{Introduction}\label{ii:sec:introduction}

A central problem in the study of the K\"ahler--Ricci flow is to
understand how its singularities reflect the topology and geometry of the underlying space. The analytic minimal model program of
Song--Tian \cite{SongTianSurfaces,SongTianMeasures,SongTianSingularities} proposes
that the flow realize the birational operations of the minimal
model program and the collapse of Fano fibrations. Its refined
formulation \cite{JST} asks, in addition, for a description of the geometry at the singular scale. In the present paper we address the scalar-curvature and fibre-diameter bounds predicted in  \cite{JST}, and the limiting geometry when the collapsing
fibre admits a K\"ahler--Einstein metric.

Let $X$ be a  projective manifold. We consider
the maximal unnormalized K\"ahler--Ricci flow
\begin{equation}\label{ii:eq:flow}
\left\{
\begin{array}{l} \partial_t\omega(t)=-\Ric(\omega(t)),\qquad 0\le t<T<\infty, \\
\\
 \omega(0)=\omega_0,
\end{array}\right.
\end{equation}
where $T$ is the first singular time and $\omega_0$ is a smooth K\"ahler metric. The finite singular time $T$ only depends on the intial K\"ahler class $[\omega_0]$. We let $g(t)$ be the K\"ahler metric associated to $\omega(t)$ and let $R_g=R_\omega=\tr_\omega\Ric(\omega)$
be the scalar curvature of $g(t)$. 

Throughout this paper, we assume that $[\omega_0]\in H^2(X, \mathbb{Q})\cap H^{1,1}(X, \mathbb{R})$. The limiting class $[\omega_0]-2\pi T c_1(X)$ is semiample by Kawamata's base-point-free theorem, and 
it determines a surjective morphism 
$$
\Phi:X\to Y
$$
with connected fibres onto a normal projective variety $Y$. The morphism describes which directions are contracted, but by itself gives little information about the rate of contraction or the curvature near the singular time. These are the questions addressed by the
following conjecture proposed in \cite[Conjecture~1.1]{JST}. 

\begin{conjecture}
\label{ii:conj:JST}
Let \eqref{ii:eq:flow} $g(t)$ be the maximal solution of the K\"ahler-Ricci flow on $X\times [0, T)$. There exists a constant $C>0$ such that
\[
 \sup_X|R_{g(t)}|\le\frac{C}{T-t},\qquad
 \sup_{q\in Y}\diam_{(X,g(t))}\Phi^{-1}(q)
       \le C\sqrt{T-t},\qquad 0\le t<T. 
\]
\end{conjecture}

Here $\Phi^{-1}(q)$ is a connected fibre whose fibre diameter is measured in the ambient manifold $X$. The fibres themselves may be singular or reducible. When $Y$ is a
point, the conjecture reduces to Perelman's scalar-curvature and
diameter estimates for the Fano Ricci flow. For a positive-dimensional
base, it asks for the same parabolic scale to control every fibre,
with a single constant. It applies both to collapsing fibrations
and to birational contractions. The bound concerns scalar curvature, whereas bounds for the full curvature tensor only hold in special cases and must fail in general. 

The compact Fano case also provides a model for the limiting geometry.
For a Fano manifold admitting a smooth K\"ahler--Einstein metric,
Tian--Zhu \cite{TianZhuFlow,TianZhuFlowII} proved that the normalized
K\"ahler--Ricci flow starting from any K\"ahler metric in $2\pi c_1$
converges smoothly, modulo holomorphic automorphisms, to a
K\"ahler--Einstein metric. Thus the Hamilton--Tian limit in this case
is smooth and retains the original complex structure. Tian--Zhang
\cite{TianZhangActa} developed a regularity theory under integral
Ricci-curvature bounds and proved the Hamilton--Tian conjecture
in complex dimension at most three, obtaining soliton limits that
are smooth away from a set of real codimension at least four.
Convergence to a K\"ahler--Einstein metric under positive bisectional
curvature and suitable stability assumptions was established by
Phong--Song--Sturm--Weinkove \cite{PSSWPositive}.

An essential step towards this conjecture was taken in \cite{JST} with the introduction of localization for the Ricci potential and its 
\emph{Ricci vertices}. For a smooth closed representative $\theta_Y$
of the limiting base class, let $u_{\theta_Y}$ satisfy
\[
 \Ric(\omega(t))+\ddc u_{\theta_Y}
   =\frac{\omega(t)-\Phi^*\theta_Y}{T-t}.
\]
A global minimum of $u_{\theta_Y}(\cdot,t)$ is called a Ricci vertex
associated to $\theta_Y$. The Li--Yau and Harnack estimates of
\cite[Theorems~1.5--1.6]{JST} give scalar-curvature and diameter
bounds on tubes of base radius comparable to $\sqrt{T-t}$ around
such a vertex, whenever a larger tube has volume at most
$C(T-t)^n$. The same work establishes complex-analytic compactness
of the associated blow-ups by partial $C^0$ estimates
\cite[Section~8]{JST}. The remaining difficulty, for the estimates
on a prescribed fibre, is to choose a representative whose vertex
lies sufficiently close to that fibre.

This paper continues \cite{Splitting}, which establishes a splitting theorem for every tangent at a fixed limiting point over a Fano bundle \cite[Theorem~1.1]{Splitting}. In the case of one-dimensional fibres, the full Type I curvature and sharp
fibre-diameter estimates are proved \cite[Theorem~1.2]{Splitting}.  The goal of this paper is to establish Conjecture \ref{ii:conj:JST} for the Fano bundle and to further extend the Type I curvature bounds from $\mathbb{P}^1$-bundle to Fano bundles with K-polystabale fibres. Related results in complex dimension two
also seen the work of Xu--Zhang \cite{XuZhang,XuZhangSurfaces}
and Cifarelli--Conlon--Hallgren--Zhang \cite{CCHZ}.

\subsection{Type I bounds for scalar curvature and fibre diameter}

%For a Fano bundle, the horizontal estimates of Fu--Zhang \cite{FuZhang} already give the required control of tube volumes. The question is therefore one of locating Ricci vertices. The key observation is that the terminal base potential has bounded real Hessian. Its quadratic lower supports allow us to place a vertex within base distance $C\sqrt{T-t}$ of any fibre, which isprecisely the scale required by \cite[Theorem~1.6]{JST}.

Let  $X$ be a projective manifold of complex dimension $n$ and 
\begin{equation}\label{finitemap}
\Phi:X\to Y
\end{equation} be the surjective morphism induced by the limiting class of the K\"ahler-Ricci flow (\ref{ii:eq:flow}). 
Throughout this paper, we assume that $\Phi$ is a \emph{Fano bundle} over a smooth projective manifold  $Y$ of complex dimension $0<m<n$. More precisely, there is a Zariski-open
cover $\{U_\alpha\}$ of $Y$ and biregular isomorphisms
\[
 \Phi^{-1}(U_\alpha)\cong U_\alpha\times F
\]
over $U_\alpha$, where $F$ is a fixed smooth Fano
manifold of complex dimension $n-m$. Write $F_q=\Phi^{-1}(q)$.

\begin{theorem}\label{ii:thm:main}  Let $\omega(t)$ be the maximal solution of the K\"ahler-Ricci flow  
\eqref{ii:eq:flow} on $X\times [0, T)$. Suppose that $\Phi: X\rightarrow Y$ is a Fano bundle and the limiting cohomology class of the flow satisfies 
\begin{equation}\label{ii:eq:class}
 [\omega_0]-2\pi T c_1(X)=\Phi^*[\omega_Y],
\end{equation}
for  a K\"ahler form    $\omega_Y$ on $Y$ . Then Conjecture~\ref{ii:conj:JST} holds. More precisely, there is a constant $C>0$
such that, for all $0\le t<T$,
\begin{equation}\label{ii:eq:main}
 \sup_X|R_{g(t)}|\le\frac C{T-t},\qquad
 \sup_{q\in Y}\diam_{(X,g(t))}F_q\le C\sqrt{T-t}.
\end{equation}
The diameter is measured by the ambient distance in $(X, g(t))$.  
\end{theorem}

There is no
stability assumption on $F$ and no symmetry assumption on the
initial metric. Theorem \ref{ii:thm:main} easily extends locally to Fano fibrations whose general fibres are biholomorphic to each other, away from fibres over a certain analytic subvariety of $Y$.

 The proof of Theorem \ref{ii:thm:main} does not rely on the recent result on the analytic structure of tangent flows in \cite{HZ}. Theorem \ref{ii:thm:main} provides the geometric tools for building the partial $C^0$-estimate along the flow (\ref{ii:eq:main}), which can be used to construct the tangent flow as a complete K\"ahler-Ricci soliton with normal singularities via Type I blowup. This will also lead to a new proof of the splitting theorem in \cite{Splitting}.

We would like to remark that there is an alternative proof for Theorem \ref{ii:thm:main} based on fibrewise
exponential normalization of the Ricci potential, which involves a horizontal volume correction. Then Berndtsson's
direct-image curvature formula combined with a localized
scalar--Sobolev iteration can be used to control the Ricci vertices, avoiding the  $C^{1,1}$ regularity used in our proof in this paper. Details of this approach will appear in a forthcoming note.

\subsection{Type I bounds for full curvature tensors}

The scalar estimate leaves open the possibility of curvature
concentration at smaller scales. When $F$ admits a smooth
K\"ahler--Einstein metric, or equivalently when $F$ is K-polystable, this possibility can be excluded, similar to the smooth convergence of Fano Ricci flow on K\'ahler-Einstein manifolds.
The following theorem gives both the curvature bound and the
corresponding description of the limiting geometry.

\begin{theorem}\label{ii:thm:einstein}
In addition to the assumptions of Theorem~\ref{ii:thm:main},
if the fibre $F$ admits a smooth K\"ahler--Einstein metric $g_F$, then
there exists a constant $C>0$ such that the following  hold.
\begin{enumerate}[label=\textup{(\arabic*)},leftmargin=2.2em]
\item For every $0\le t<T$,
\begin{equation}\label{ii:eq:typeI}
 \sup_X|\Rm(g(t))|_{g(t)}\le\frac C{T-t}.
\end{equation}
\item For every $q\in Y$ and $0\le t<T$,
\begin{equation}\label{ii:eq:intrinsic-diameter}
 C^{-1}\sqrt{T-t}\le
 \diam_{(F_q,g(t)|_{F_q})}F_q\le C\sqrt{T-t},
\end{equation}
where the diameter is intrinsic, measured by the induced metric on
$F_q$.

\item\label{ii:thm:einstein-limits}
For any $q_i\in Y$, $x_i\in F_{q_i}$ and $\tau_i\searrow0$, the
pointed rescaled K\"ahler--Ricci flows
\[
 \left(X,J,\tau_i^{-1}g(T+\tau_i a),x_i\right),
 \qquad -T/\tau_i\le a<0,
\]
admit a subsequence converging in the pointed smooth
Cheeger--Gromov sense on compact spacetime subsets of $a<0$.
Every such limit is holomorphically isometric to
\begin{equation}\label{ii:eq:einstein-model}
 \left(\C^m\times F, \
 g_{\mathrm{Euc}}+(-a)g_F\right),
\end{equation}
 where $g_{\mathrm{Euc}}$ is the Euclidean Kahler metric on $\mathbb{C}^m$. 
\end{enumerate}
\end{theorem}

The case of projective-space fibres gives an immediate application.

\begin{corollary}[Projective bundles]\label{ii:cor:projective}
In addition to the assumptions of Theorem~\ref{ii:thm:main}, suppose that
$F\cong\P^r$ for some $r\ge1$. Then there is a constant $C>0$
such that
\[
 \sup_X|\Rm(g(t))|_{g(t)}\le\frac{C}{T-t},
 \qquad 0\le t<T.
\]

\end{corollary}

In particular, this conclusion holds for $X=\P(E)$, where $E$ is a
holomorphic vector bundle of rank $r+1$ over a connected smooth
projective manifold $Y$ of positive dimension.
 
Theorem \ref{ii:thm:einstein} extends to general Fano bundles (c.f. \cite{JSfb3}) in the following sense. If the fibre $F$ admits a smooth K\"ahler-Ricci soliton metric, then the Type I curvature estimates still hold as in Theorem \ref{ii:thm:einstein}. For a general fibre, the Type I curvature estimates fail in general, however, the Type I scalar curvature and fibre diameter estimate from Theorem \ref{ii:thm:main} can be used to produce a splitting on the Type I blow-up limit, where the compact component is a compact K\"ahler-Ricci soliton space biholomorphic to the algebraic optimal degneration of the original fibre.

\begin{remark}\label{ii:rem:local}
Theorem~\ref{ii:thm:main} and Theorem \ref{ii:thm:einstein} below also hold locally over the
regular product locus of an isotrivial Fano fibration in the
finite-time AMMP setting. More precisely, let
$Y^\circ\subset Y_{\mathrm{reg}}$ consist of regular values of
$\Phi$ and be covered by Zariski-open product regions with common
fibre $F$. For every $K\Subset Y^\circ$, the estimates hold on
$\Phi^{-1}(K)$ with constants depending on $K$. The full curvature,
intrinsic diameter and smooth limiting-product conclusions retain
the K\"ahler--Einstein hypothesis. The limiting statement holds at
fixed $q\in Y^\circ$ and for all moving sequences $q_i\in K$.
The regular-value condition is part of the fibration hypothesis;
smoothness of $Y$ alone does not exclude singular fibres. The local
splitting input is \cite[Remark~1.1 and Proposition~2.1]{Splitting}.
\end{remark}

%\begin{remark}\label{ii:rem:irrational}
%When $Y$ is smooth, $\theta_Y$ may be any K\"ahler form on $Y$
%satisfying
%\[
% [\omega_0]-2\pi T c_1(X)=\Phi^*[\theta_Y].
%\]
%One may take $\omega_Y=\theta_Y$ throughout; no rationality of
%$[\theta_Y]$ is required. The morphism $\Phi$ and the stated
%Zariski-open-product property are assumed. An auxiliary projective
%embedding of $Y$ is chosen independently of $[\theta_Y]$, and the
%polarization of each fibre is its anticanonical polarization.
%\end{remark}

The strategy for our scalar-curvature and ambient fibre-diameter
estimates originates in our joint work with Tian \cite{JST}.
A central idea of that work is to extend Perelman's estimates
by working near Ricci vertices and exploiting the freedom to
change the smooth representative of the limiting class.
The Li--Yau and Harnack estimates established there turn suitable
vertex localization and a tube-volume bound into Type~I
scalar-curvature and diameter estimates
\cite[Theorems~1.5--1.6]{JST}.

Our argument implements this strategy uniformly at every fibre.
The $C^{1,1}$ regularity of the limiting base potential provides
smooth quadratic wells, which select representatives whose Ricci
vertices lie within base distance $C\sqrt{T-t}$ of any prescribed
fibre. The horizontal estimates of Fu--Zhang \cite{FuZhang},
together with the fibre-volume identity, verify the tube-volume
hypothesis. The desired scalar-curvature and ambient diameter
bounds then follow from \cite[Theorem~1.6]{JST}. Thus the new
localization argument supplies the geometric input needed to
apply the analytic estimates of \cite{JST} throughout the bundle.

\bigskip

Section~\ref{ii:sec:preliminaries} is a collection of notations and known estimates.
Sections~\ref{ii:sec:horizontal}--\ref{ii:sec:scalar} establish the
horizontal estimates, localization of Ricci potentials with Ricci vertices and the proof of 
Theorem~\ref{ii:thm:main}. Section~\ref{ii:sec:relative} gives the
relative potentials and the fibrewise Poincar\'e inequality.
Section~\ref{ii:lim:sec:geometry} establishes the limiting measures,
anticanonical convergence and entropy continuity.
Section~\ref{ii:sec:curvature} identifies the Einstein limits and completes the proof of Theorem~\ref{ii:thm:einstein}.

\bigskip

\noindent\textbf{Acknowledgments.}
The authors thank Professor Gang Tian for his constant encouragement and support. The authors used AI-assisted tools (ChatGPT 5.6 Sol) during the preparation of this manuscript, most notably in the development of Section 5. All mathematical content remains the sole responsibility of the authors.

\section{Preliminaries}\label{ii:sec:preliminaries}

We collect the notation and estimates that will be used throughout
the paper. The parabolic Monge--Amp\`ere equation fixes the
normalization of the potentials. The Schwarz estimate and the
horizontal estimate of Fu--Zhang then control the horizontal
quotient metric and the volumes of collapsing tubes. These are
the starting points for the horizontal Hessian estimates in
Section~\ref{ii:sec:horizontal} and the Ricci-vertex construction
in Section~\ref{ii:sec:vertices}.

\subsection{The flow and its scalar equation}\label{ii:pre:flow}

We retain the setup of Theorem~\ref{ii:thm:main}, with
$\Phi:X^n\to Y^m$, whose fibres $F_q=\Phi^{-1}(q)\cong F$ have complex dimension $n-m$.
Replacing $\omega(t)$ by $T^{-1}\omega(Tt)$ and $\omega_Y$ by
$T^{-1}\omega_Y$, we first normalize the maximal time to $T=1$.
The class identity \eqref{ii:eq:class} then gives
$[\omega_0-\Phi^*\omega_Y]=2\pi c_1(X)$. By the
$\partial\bar\partial$-lemma, there is a smooth positive volume
form $\Omega$ satisfying
\[
 \Ric(\Omega)=\omega_0-\Phi^*\omega_Y.
\]
Fix its multiplicative constant. With
$\widehat\omega_t=(1-t)\omega_0+t\Phi^*\omega_Y$, the flow is
\begin{equation}\label{ii:eq:MA}
 \omega(t)=\widehat\omega_t+\ddc\phi(t),\qquad
 \frac{\partial\phi}{\partial t}
 =\log\frac{(\widehat\omega_t+\ddc\phi)^n}{\Omega},
 \qquad \phi(0)=0.
\end{equation}
The scalar equation fixes the time-dependent additive constant
of $\phi$. Indeed, the $\ddc$ of the difference between the two
sides of the scalar evolution equation vanishes; that difference
is spatially constant and is removed by this normalization.
For the normalized flow and its Ricci potential we set
\[
 \tau=1-t=e^{-s},\qquad \widetilde{\omega}(s)=\tau^{-1}\omega(t),
\]
\begin{equation}\label{ii:eq:Uintro}
 U=\tau\dot\phi+\phi+nt,\qquad u_0=e^sU.
\end{equation}

We use the complex conventions
\[
 \Delta_\omega v=g^{i\bar j}v_{i\bar j},\qquad
 |\partial v|_\omega^2=g^{i\bar j}v_i v_{\bar j},\qquad
 R_\omega=\tr_\omega\Ric(\omega),\qquad
 dV_\omega=\frac{\omega^n}{n!}.
\]
If $g_\omega=2\operatorname{Re}(g_{i\bar j}dz^i\otimes d\bar z^j)$
is the associated real metric, our real Ricci-flow convention is
$g=2g_\omega$. The unscaled associated metric satisfies
\[
 \Delta_{g_\omega}=2\Delta_\omega,\qquad
 R_{g_\omega}=2R_\omega.
\]
Thus, for real functions $v$,
\begin{equation}\label{ii:eq:real}
 \begin{gathered}
 \partial_tg=-2\Ric(g),\qquad \Delta_g=\Delta_\omega,
 \qquad R_g=R_\omega,\\
 |\nabla v|_g^2=|\partial v|_\omega^2,\qquad
 dV_g=2^n dV_\omega.
 \end{gathered}
\end{equation}
Constant rescaling leaves the Ricci tensor unchanged and
multiplies the volume form in real dimension $2n$ by $2^n$,
which explains the normalization.
In particular, the complex and real heat operators agree. We write
\begin{equation}\label{ii:eq:boxes}
 \Box_\omega=\partial_t-\Delta_{\omega(t)},\qquad
 \Box_{\widetilde{\omega}}=\partial_s-\Delta_{\widetilde{\omega}(s)},\qquad
 \Box_g=\partial_t-\Delta_{g(t)}=\Box_\omega.
\end{equation}
When we return to the original maximal time, we use
$\tau=T-t$, $s=-\log\tau$ and $\widetilde{\omega}=\tau^{-1}\omega(t)$.
In either convention,
\[
 \partial_s=\tau\partial_t,\qquad
 \Delta_{\widetilde{\omega}}=\tau\Delta_\omega.
\]

Constants $c,C>0$ may change from line to line and depend only on
the fixed flow and bundle. Local constants may also depend on a
fixed compact subset of the regular product locus. We work in a
finite collection of inner product charts, with larger charts
reserved for cutoffs. The notation $B_{\omega_Y}(q,r)$ denotes a base
metric ball; $B_r(q)$ denotes a Euclidean ball in fixed coordinates.
Fibre diameters are ambient unless an induced fibre metric is
explicitly specified.

\subsection{Basic potential and horizontal estimates}\label{ii:pre:estimates}

The following estimates are known. The potential bounds and the
parabolic Schwarz lemma apply to the flow, while the bundle
structure supplies the complementary horizontal upper bound of
Fu--Zhang. We state their precise forms and record the dependence
of the constants needed for later applications of the estimates
of Jian--Song--Tian.

\begin{lemma}\label{ii:lem:basic}
There is a uniform constant $C$ such that
\begin{equation}\label{ii:eq:basic}
 |\phi|\le C,\quad \dot\phi\le C,\quad R_\omega\ge-C,
 \quad |U|\le C,\quad\omega(t)\ge C^{-1}\Phi^*\omega_Y.
\end{equation}
Moreover,
\begin{equation}\label{ii:eq:Uident}
 U_t=n-\tau R_\omega,\quad
 \ddc U=\omega-\Phi^*\omega_Y-\tau\Ric(\omega),\quad
 \Box_{\omega} U=\tr_\omega\Phi^*\omega_Y.
\end{equation}
\end{lemma}
The potential and Schwarz bounds in \eqref{ii:eq:basic} are
\cite[Lemma~2.1]{FuZhang}; the differential Schwarz estimate is
\cite[Lemma~4.1]{JST}. The scalar lower bound and the estimate
$|U|\le C$ are recalled in \cite[proofs of Lemmas~4.2 and~4.4]{JST}.
The identities \eqref{ii:eq:Uident} are
\cite[equations~(4.14)--(4.15)]{JST}, written under the change
\[
 u_0=e^sU,\qquad \partial_s=\tau\partial_t,\qquad
 \Delta_{\widetilde{\omega}}=\tau\Delta_\omega.
\]
Here $u_0$ denotes the unperturbed normalized potential in that paper.

\begin{lemma}\label{ii:lem:horizontal}
Let $B$ be a coordinate ball contained in a Zariski trivializing open set, so that $\Phi^{-1}(B)\simeq B\times F$. Every relatively compact inner base ball $B'\Subset B$ satisfies
\begin{equation}\label{ii:eq:rawhorizontal}
 \omega(t)|_{B'\times\{z\}}\le C\omega_Y|_{B'}
 \qquad(z\in F,\ 0\le t<1).
\end{equation}
The constants are uniform in a fixed finite collection of inner product charts.
\end{lemma}
This is the horizontal estimate of Fu--Zhang
\cite[Lemma~2.2]{FuZhang}, also stated in
\cite[Lemma~9.1]{JST}, in the compact-local form recorded in
\cite[Proposition~2.1]{Splitting}. On each relatively compact inner
product chart, the fixed weight in that estimate is bounded away
from zero. A finite cover makes the constants uniform.

\begin{remark}\label{ii:rem:analytic-inputs}
The analytic results of \cite{JST} apply with the class identity
\eqref{ii:eq:class}. We explain the change from the projective
reference form in that paper. When $Y$ is smooth, fix a projective
embedding independently of $[\omega_Y]$, and let $\eta$ be its induced
Fubini--Study form. Compactness gives
$c\eta\le\omega_Y\le C\eta$. The Schwarz estimate for $\Phi^*\omega_Y$
therefore controls the pullbacks of every fixed auxiliary base form.
For a uniformly $C^4$-bounded family of smooth perturbations $\rho$,
choose a fixed sufficiently large $A$ so that
\[
 \kappa_\rho=A\omega_Y+\ddc\rho\ge\omega_Y,\qquad
 \ddc\rho-\omega_Y=\kappa_\rho-(A+1)\omega_Y.
\]
Both $\kappa_\rho$ and $\omega_Y$ have uniform positive lower bounds
and upper bounds for their holomorphic bisectional curvatures.
The differential Schwarz estimate \cite[Lemma~4.1]{JST} applied to
these two metrics therefore gives precisely the trace inequalities
used in \cite[Section~5.1, equations~(5.6)--(5.9)]{JST}.
The potential equations are exactly \eqref{ii:eq:Uident} and its
normalized version, so the cited estimates use no rationality of
$[\omega_Y]$. A fixed smooth extension from the embedded compact base
bounds the ambient $C^4$ norms of any uniformly $C^4$-bounded family
of perturbations. This supplies the uniform constants required by
\cite[Proposition~5.2 and Theorem~1.6]{JST}.

For the local formulation in Remark~\ref{ii:rem:local}, all
variable cutoffs and perturbations have derivatives supported in
fixed relatively compact regular product charts and are constant
outside larger charts. Their pullbacks are smooth on $X$, and the
same Schwarz and derivative bounds hold with constants depending on
these charts. The ambient projective reference supplies the global
Schwarz bound when the base is singular.
\end{remark}

\subsection{Quotient metrics and fibre volumes}\label{ii:pre:quotient}

The estimates just recalled have two consequences that will be
used repeatedly: uniform control of the horizontal quotient and
an upper bound for the volume of a collapsing tube. We collect
the estimates in two lemmas and record the algebraic identities
separately.

For $x\in X$ and $\xi\in T_{\Phi(x)}Y$, define the horizontal
quotient metric by
\[
 H_t(x)(\xi,\bar\xi)
 =\inf_{d\Phi_x(V)=\xi}\omega(t)_x(V,\bar V).
\]
Its normalized version and determinant term are
\begin{equation}\label{ii:eq:quotient}
 H=\tau^{-1}H_t,\qquad
 D=\log\frac{\det H_t}{\det\omega_Y}
   =\log\frac{\det H}{e^{ms}\det\omega_Y},\qquad
 D_q=D|_{F_q}.
\end{equation}
The determinant ratio is intrinsic on the base tangent space at
$\Phi(x)$, so $D$ is a well-defined function on $X$.

\begin{lemma}[Horizontal quotient]\label{ii:pre:lem:quotient}
There are uniform constants $c,C>0$ such that
\begin{equation}\label{ii:eq:SQ}
 c\omega_Y\le H_t\le C\omega_Y,
\end{equation}
\begin{equation}\label{ii:pre:eq:normalized-quotient}
 ce^s\omega_Y\le H\le Ce^s\omega_Y,\qquad |D|\le C.
\end{equation}
\end{lemma}
\begin{proof}
The Schwarz bound in Lemma~\ref{ii:lem:basic} applies to every
lift of a base vector, and gives the lower bound in
\eqref{ii:eq:SQ}. The product-horizontal lift and
Lemma~\ref{ii:lem:horizontal} give the upper bound.
Scaling \eqref{ii:eq:SQ} gives the bounds for $H$, while
\[
 c^m\le\frac{\det H_t}{\det\omega_Y}\le C^m
\]
gives the bound for $D$.
\end{proof}

In product coordinates, with fibre coordinates listed first, write
\[
 \bigl(g_{I\bar J}(t)\bigr)
   =\begin{pmatrix}A&B\\ B^*&E\end{pmatrix}.
\]
Here this matrix consists of the Hermitian coefficients of
$\omega(t)$. Then
\begin{equation}\label{ii:lim:eq:schur}
 H_t=E-B^*A^{-1}B,\qquad
 \det \bigl(g_{I\bar J}(t)\bigr)=(\det A)(\det H_t).
\end{equation}
For every smooth real function $a$ on the base,
\begin{equation}\label{ii:pre:eq:base-contractions}
 \begin{aligned}
 |\partial\Phi^*a|_{\widetilde{\omega}}^2
   &=(H^{-1})^{\alpha\bar\beta}a_\alpha a_{\bar\beta},\\
 \tr_{\widetilde{\omega}}\Phi^*\omega_Y&=\tr_H\omega_Y.
 \end{aligned}
\end{equation}

For a lift $(\zeta,\xi)$, completing the square gives
\[
 \begin{pmatrix}\zeta\\\xi\end{pmatrix}^{\!*}
 \begin{pmatrix}A&B\\ B^*&E\end{pmatrix}
 \begin{pmatrix}\zeta\\\xi\end{pmatrix}
 = (\zeta+A^{-1}B\xi)^*A(\zeta+A^{-1}B\xi)
   +\xi^*(E-B^*A^{-1}B)\xi.
\]
Taking the infimum in $\zeta$ gives explicitly
\[
 \inf_\zeta\left|(\zeta,\xi)\right|_{\omega(t)}^2
 =\xi^*H_t\xi,\qquad \zeta_{\min}=-A^{-1}B\xi.
\]
The corresponding matrix factorization is
\[
 \begin{pmatrix}A&B\\ B^*&E\end{pmatrix}
 =\begin{pmatrix}I&0\\B^*A^{-1}&I\end{pmatrix}
  \begin{pmatrix}A&0\\0&H_t\end{pmatrix}
  \begin{pmatrix}I&A^{-1}B\\0&I\end{pmatrix}.
\]
The two triangular factors have determinant one, so
$\det\bigl(g_{I\bar J}(t)\bigr)=(\det A)(\det H_t)$.
Block inversion shows that the horizontal block of the inverse
normalized metric is $H^{-1}$, proving
\eqref{ii:pre:eq:base-contractions}.

For the fibre geometry we use the induced normalized metric
$h_q$ and its probability measure $\nu_{q,s}$:
\begin{equation}\label{ii:pre:eq:fibre-notation}
 h_q(s)=\widetilde{\omega}(s)|_{F_q},\qquad
 d\nu_{q,s}=\frac{h_q(s)^{n-m}}{\int_{F_q}h_q(s)^{n-m}}.
\end{equation}
For a probability measure $\nu$ and a real function $u\in L^2(\nu)$,
our mean and variance conventions are
\[
 \bar u_\nu=\int u\,d\nu,\qquad
 \Var_\nu u:=\int\left|u-\int u\,d\nu\right|^2\,d\nu.
\]
This function variance is distinct from Bamler's metric variance
\[
 \Var_g(\mu):=\iint d_g(x,y)^2\,d\mu(x)\,d\mu(y).
\]

The normalized fibre metrics satisfy
\begin{equation}\label{ii:eq:vol}
 [h_q(s)]=2\pi c_1(F_q),\qquad
 \Vol_{\omega(t)|_{F_q}}F_q=\tau^{n-m} V_F.
\end{equation}
Their complex and real volumes are independent of $q,s$ and equal to
\begin{equation}\label{ii:lim:eq:volumes}
 V_F=\frac1{(n-m)!}\int_{F_q}h_q^{n-m}
     =\frac{(2\pi)^{n-m}}{(n-m)!}\int_F c_1(F)^{n-m},
 \qquad V=2^{n-m}V_F.
\end{equation}
Here $V$ is the volume of the real metric $2g_{h_q}$, and
\[
 d\nu_{q,s}
 =\frac{dV_{h_q}}{V_F}
 =\frac{dV_{\omega(t)|_{F_q}}}
        {\Vol_{\omega(t)|_{F_q}}F_q}.
\]
The normal bundle of $F_q$ is the trivial bundle
$F_q\times T_qY$, so adjunction gives
$c_1(X)|_{F_q}=c_1(F_q)$. Restriction of the class identity
\eqref{ii:eq:class} yields
$[\omega(t)|_{F_q}]=2\pi\tau c_1(F_q)$.
This proves \eqref{ii:eq:vol} and the complex-volume formula.
The real-volume formula follows from \eqref{ii:eq:real}; constant
rescaling also leaves the normalized fibre measure unchanged.

\begin{lemma}[Tube volume bound]\label{ii:pre:lem:volumes}
For all sufficiently small $r$, uniformly in $q$ and $t$,
\begin{equation}\label{ii:eq:tubevolume}
 \Vol_{\omega(t)}\Phi^{-1}(B_{\omega_Y}(q,r))
       \le C\tau^{n-m} r^{2m}.
\end{equation}
In the local setting this bound holds uniformly on each
fixed compact subset of the regular product locus.
\end{lemma}
\begin{proof}
The tube bound is \cite[Lemma~9.2 and equation~(9.5)]{JST}.
To use it with $\omega_Y$, choose an auxiliary projective embedding
of $Y$. On each buffered product chart its induced metric is
uniformly equivalent to $\omega_Y$, and the divisor weight in the
cited estimate is bounded away from zero. An $\omega_Y$-ball is
contained in an auxiliary-metric ball of a fixed multiple of its
radius, which proves \eqref{ii:eq:tubevolume}. A finite cover
gives uniform constants. The same argument applies on compact
subsets of the regular product locus.
\end{proof}

\section{Horizontal Hessian estimates}\label{ii:sec:horizontal}

The location of a Ricci vertex is governed by the terminal base
potential. For this purpose a bound for its complex Hessian is
not sufficient: a quadratic lower support requires control of the
full real Hessian. This additional control comes from a simple
feature of the bundle coordinates, namely that horizontal
derivatives of a suitably chosen potential satisfy the heat
equation. In Section~\ref{ii:sec:terminal}, these estimates will
give a $C^{1,1}$ terminal potential on the base. Throughout this section we use the
normalizations and known estimates of
Section~\ref{ii:sec:preliminaries}, with $T=1$.

In a product chart, the fixed fibre geometry can be absorbed into
the reference forms of the potential equation. The reference
density is then independent of the base variables, and horizontal
differentiation gives an exact heat equation. The Bochner formula
along the flow turns this identity into the real Hessian estimate
of Proposition~\ref{ii:prop:horizontalhessian}; the cancellation
of the Ricci terms is essential here.

\begin{lemma}\label{ii:lem:heat}
On the product of a sufficiently small coordinate ball $B$ with $F$, there are fixed smooth real functions $\psi_0,H_0$, a K\"ahler form $\omega_F\in2\pi c_1(F)$, and a positive volume form $\Omega_F$ with $\Ric(\Omega_F)=\omega_F$, such that
\begin{equation}\label{ii:eq:localMA}
 \Psi=\phi+\psi_0+tH_0,\quad
 \omega(t)=\tau\omega_F+\omega_E+\ddc\Psi,\quad
 \dot\Psi=\log\frac{\omega(t)^n}{\Omega_F\wedge\dd V_E}.
\end{equation}
Here $\omega_E$ is a fixed Euclidean K\"ahler form on $B$ and $\dd V_E=\omega_E^m/m!$. For every real base coordinate $y^a$,
\begin{equation}\label{ii:eq:horizontalheat}
 \Box_{\omega} \partial_{y^a}\Psi=0.
\end{equation}
\end{lemma}
\begin{proof}
Fix a K\"ahler representative $\omega_F$ of $2\pi c_1(F)$. The $\partial\bar\partial$-lemma on $F$ \cite[Chapter~VI, Lemma~8.6]{Demailly} gives a positive volume form $\Omega_F$ with the required Ricci form. Pull these objects back to $B\times F$. The inclusion of one fibre is a homotopy equivalence because $B$ is contractible. The restriction of $\omega_0-\omega_F-\omega_E$ to that fibre has zero cohomology class by \eqref{ii:eq:vol}; hence this closed real $(1,1)$-form is exact on $B\times F$.

Kodaira vanishing gives
$H^1(F,\mathcal O_F)=H^1(F,K_F\otimes(-K_F))=0$
\cite[Chapter~VII, Theorem~3.3(a)]{Demailly}.
The coherent cohomology groups of the compact manifold $F$ are
finite dimensional and Hausdorff, so the completed tensor product
in the product-cohomology formula agrees here with the ordinary
tensor product.
Since $B$ is Stein, \cite[Chapter~IX, Corollary~5.22(a);
Chapter~IV, (6.16)]{Demailly} gives
\[
 H^{0,1}(B\times F)
 \simeq H^1(B\times F,\mathcal O)
 \simeq\mathcal O(B)\otimes H^1(F,\mathcal O_F)=0.
\]
To obtain a real potential, write the exact real $(1,1)$-form as
$\dd\alpha$ with $\alpha$ real. Its $(0,2)$-component gives
$\bar\partial\alpha^{0,1}=0$. The vanishing just established gives
$\alpha^{0,1}=\bar\partial b$. By reality,
$\alpha^{1,0}=\partial\bar b$, and hence
\[
 \dd\alpha=\partial\bar\partial b+\bar\partial\partial\bar b
 =\partial\bar\partial(b-\bar b)=\ddc(2\operatorname{Im}b).
\]
Thus there is a fixed smooth real function $\psi_0$ with
\[
 \omega_0=\omega_F+\omega_E+\ddc\psi_0.
\]
Set $H_0=\log\bigl(\Omega/(\Omega_F\wedge\dd V_E)\bigr)$. Since $\dd V_E$ is flat,
\[
 \ddc H_0=-\Ric(\Omega)+\Ric(\Omega_F\wedge\dd V_E)
             =\Phi^*\omega_Y-\omega_0+\omega_F.
\]
It follows that
\begin{align*}
 \tau\omega_F+\omega_E+\ddc(\phi+\psi_0+tH_0)
 &=\tau\omega_F+\omega_E+\ddc\phi
       +\omega_0-\omega_F-\omega_E\\
 &\quad+t(\Phi^*\omega_Y-\omega_0+\omega_F)\\
 &=\tau\omega_0+t\Phi^*\omega_Y+\ddc\phi=\omega(t),\\
 \dot\Psi&=\dot\phi+H_0
       =\log\frac{\omega(t)^n}{\Omega_F\wedge\dd V_E}.
\end{align*}
This proves \eqref{ii:eq:localMA} with all constants fixed in time.

In product coordinates $(z,y)$ the reference metric $\tau\omega_F+\omega_E$ and reference volume $\Omega_F\wedge\dd V_E$ have coefficients independent of $y$. For $f=\partial_{y^a}\Psi$, differentiation of the last equation gives
\[
 f_t=g^{I\bar J}\partial_{y^a}g_{I\bar J}
     =g^{I\bar J}\partial_I\partial_{\bar J}f=\Delta_\omega f.
\]
The capital indices range over all fibre and base coordinates, including the mixed terms $g^{i\bar\beta}f_{i\bar\beta}$ and $g^{\alpha\bar j}f_{\alpha\bar j}$.
\end{proof}

\begin{proposition}\label{ii:prop:horizontalhessian}
On each fixed inner product chart,
\begin{equation}\label{ii:eq:hessian}
 |D_y\phi|+|D_y^2\phi|\le C,\qquad
 |\partial(\partial_{y^a}\Psi)|_\omega\le C.
\end{equation}
The first two norms are Euclidean norms of the real base derivatives in the fixed product coordinates.
\end{proposition}
\begin{proof}
Choose nested base balls with compact inclusions and always take their products with the entire compact fibre. The bounds for $\phi$ and the fixed corrections in \eqref{ii:eq:localMA} give $|\Psi|\le C$ on each of these products. For fixed $z\in F$, restriction to the horizontal slice gives
\[
 0<g_{E,\alpha\bar\beta}+\Psi_{\alpha\bar\beta}
      \le Cg_{E,\alpha\bar\beta}
\]
by Lemma~\ref{ii:lem:horizontal}. Taking the Euclidean trace and recalling that the real Euclidean Laplacian is four times the complex trace in standard coordinates, we obtain $|\Delta_{\R^{2m},y}\Psi|\le C$. The interior Euclidean elliptic estimate, followed by the Sobolev embedding for any fixed $p>2m$, gives
\[
 \|D_y\Psi(\cdot,z,t)\|_{L^\infty(B_1)}
 \le C\bigl(\|\Psi(\cdot,z,t)\|_{L^p(B_2)}
       +\|\Delta_{\R^{2m},y}\Psi(\cdot,z,t)\|_{L^p(B_2)}\bigr)
 \le C
\]
whenever $B_1\Subset B_2$ are two of the fixed balls; see \cite[Chapters~7 and~9]{GilbargTrudinger}. The operator in this estimate is the fixed Euclidean Laplacian. Its constants are independent of $z$ and $t$, so the first horizontal derivatives are uniformly bounded.

Fix a real base coordinate and put
\[
 f=\partial_{y^a}\Psi,\quad E=|\partial f|_\omega^2,\quad
 \mathcal H=|\nabla\partial f|_\omega^2+
             |\nabla\bar\partial f|_\omega^2.
\]
Here the two summands in $\mathcal H$ are the squared norms of the $(2,0)$ and $(1,1)$ covariant Hessians of $f$. The heat equation and the evolving Bochner formula give
\begin{align*}
 \Box_{\omega}  E&=2\operatorname{Re}\langle\partial(\Box_{\omega}  f),\partial f\rangle_\omega
       -\mathcal H=-\mathcal H,\\
 \Box_{\omega} (f^2)&=2f\Box_{\omega}  f-2|\partial f|_\omega^2=-2E.
\end{align*}
The evolving Bochner identity used here is the unnormalized form of the first identity in \cite[proof of Lemma~5.1]{JST}.
At normal coordinates,
\[
 E_\ell=\sum_i(f_{i\ell}f_{\bar i}+f_i f_{\bar i\ell}),
 \qquad |\partial E|_\omega^2\le2E\mathcal H.
\]

Take a fixed smooth base cutoff $0\le\chi\le1$ supported inside a ball on which $|f|\le C$, with $\chi=1$ on a smaller ball. The signed base Hessian of $\chi^2$ and the tensor $\partial\chi\otimes\bar\partial\chi$ are bounded by fixed multiples of $\omega_Y$. The trace estimate in Lemma~\ref{ii:lem:basic} therefore gives
\[
 |\partial\chi|_\omega^2+|\Delta_\omega\chi^2|\le C.
\]
The product rule, the last gradient estimate, and Young's inequality yield
\begin{align*}
 \Box_{\omega} (\chi^2E)
 &=-\chi^2\mathcal H-E\Delta_\omega\chi^2
      -4\chi\operatorname{Re}\langle\partial\chi,\partial E\rangle_\omega\\
 &\le-\chi^2\mathcal H+C E
          +4\sqrt2\,\chi|\partial\chi|_\omega\sqrt{E\mathcal H}\\
 &\le-\tfrac12\chi^2\mathcal H+
       \bigl(C+16|\partial\chi|_\omega^2\bigr)E
 \le-\tfrac12\chi^2\mathcal H+C_0E.
\end{align*}
Choose $K$ so that $2K>C_0$ and set $Q=\chi^2E+Kf^2$. We then have
\[
 \Box_{\omega}  Q\le-\tfrac12\chi^2\mathcal H-(2K-C_0)E\le0.
\]
For any $t_*<1$, apply the maximum principle on the closed product of the support ball with $F$ and the time interval $[0,t_*]$. On its spatial boundary $\chi=0$, so $Q=Kf^2\le C$; there is no fibre boundary. At $t=0$, $Q$ is uniformly bounded because the initial data and all reference corrections are fixed and smooth. Consequently
\[
 \sup Q\le
 \max\left\{\sup Q(\cdot,0),\,
       K\sup_{\text{spatial boundary}\times[0,t_*]}f^2\right\}\le C.
\]
The right-hand side is independent of $t_*$. Letting $t_*\uparrow1$ and using $\chi=1$ proves $E\le C$ on the smaller product.

Finally, the fixed real product-horizontal vector fields $V_b=\partial_{y^b}$ have bounded lengths in $g$ by Lemma~\ref{ii:lem:horizontal}. By \eqref{ii:eq:real},
\[
 |\partial_{y^b}\partial_{y^a}\Psi|
 =|\dd f(V_b)|\le|\nabla f|_g|V_b|_g
 =|\partial f|_\omega|V_b|_g\le C.
\]
Subtracting the fixed derivatives of $\psi_0+tH_0$ proves both base derivative estimates in \eqref{ii:eq:hessian}.
\end{proof}

\section{The limiting base potential}\label{ii:sec:terminal}

The potential becomes constant on each fibre as the singular time
is approached. The horizontal Hessian estimate shows that the
resulting function on the base is $C^{1,1}$, and monotonicity
upgrades pointwise convergence to uniform convergence. A one-sided
comparison with this terminal value will allow the minimum
principle to retain information at the scale $1-t$.

\begin{proposition}\label{ii:prop:terminal}
There is $\psi_\infty\in C^{1,1}(Y)$ such that, uniformly on $X$,
\begin{equation}\label{ii:eq:terminal}
 \phi(t)\longrightarrow\Phi^*\psi_\infty,\qquad
 U(t)\longrightarrow\Phi^*\psi_\infty+n.
\end{equation}
Moreover,
\begin{equation}\label{ii:eq:terminaloneside}
 U\ge\Phi^*\psi_\infty+nt-C\tau^2.
\end{equation}
\end{proposition}
\begin{proof}
By \cite[Lemma~2.1(3) and the proof of Theorem~1.1]{FuZhang},
$\phi(t)$ has a bounded pointwise limit constant on every fibre.
Write this limit as $\Phi^*\psi_\infty$. Choose $C_\phi$ with
$\dot\phi\le C_\phi$, so that $\phi(t)-C_\phi t$ decreases.
We establish the additional $C^{1,1}$ regularity and the convergence of $U$.

Fix a local product section $y\mapsto(y,z_0)$. Proposition~\ref{ii:prop:horizontalhessian} bounds the restrictions of $\phi(t)$, their first real derivatives, and their second real derivatives on every inner ball. Thus their gradients form a uniformly bounded and equicontinuous family. Arzel\`a--Ascoli gives $C^1$ convergence along a subsequence of any sequence $t_j\uparrow1$. Pointwise convergence identifies every subsequential function limit with $\psi_\infty$.
To identify its derivative as well, let the gradients of such a
subsequence converge uniformly to a vector field $V$ on a smaller
closed coordinate ball. For a segment contained in this ball,
\begin{align*}
 \psi_\infty(y')-\psi_\infty(y)
 &=\lim_j\bigl(\phi(y',z_0,t_j)-\phi(y,z_0,t_j)\bigr)\\
 &=\lim_j\int_0^1
 D_y\phi(y+a(y'-y),z_0,t_j)\cdot(y'-y)\,\dd a\\
 &=\int_0^1V(y+a(y'-y))\cdot(y'-y)\,\dd a.
\end{align*}
Uniform convergence justifies the last passage. Dividing by the
length of a coordinate segment and letting that length tend to zero
shows $D_y\psi_\infty=V$. In particular, the derivative limit is
independent of the chosen subsequence. Moreover, the estimate
\[
 |D_y\phi(y,z_0,t)-D_y\phi(y',z_0,t)|\le C|y-y'|
\]
passes to such a subsequential limit. Therefore $\psi_\infty$ is $C^{1,1}$ on the inner ball. The descriptions agree on overlaps because the pointwise limit is constant on each fibre, and a finite cover gives $\psi_\infty\in C^{1,1}(Y)$.

The functions $\phi(t)-C_\phi t$ are continuous on compact $X$, decrease, and have the continuous limit $\Phi^*\psi_\infty-C_\phi$. Dini's theorem along any increasing sequence $t_j\uparrow1$ gives uniform convergence on that sequence; monotonicity sandwiches intermediate times between consecutive sequence values. Hence $\phi(t)\to\Phi^*\psi_\infty$ uniformly as $t\uparrow1$.

Let $C_R\ge0$ satisfy $R_\omega\ge-C_R$. Set $W_0=\phi+\tau\dot\phi=U-nt$. It is uniformly bounded by Lemma~\ref{ii:lem:basic}. Direct differentiation gives
\[
 (W_0)_t=\tau\ddot\phi=-\tau R_\omega,
 \qquad
 \partial_t\left(W_0+\frac{C_R}{2}\tau^2\right)
       =-\tau(R_\omega+C_R)\le0.
\]
Thus $W_0$ has a bounded pointwise limit $W_T$. At each fixed $x$, the difference $\tau\dot\phi=W_0-\phi$ converges to $a(x)=W_T(x)-\phi_T(x)$. If $a(x)>0$, then for $t$ sufficiently close to one,
\[
 \dot\phi(x,t)\ge\frac{a(x)}{2(1-t)},\qquad
 \phi(x,t)-\phi(x,t_0)
       \ge\frac{a(x)}2\log\frac{1-t_0}{1-t}\longrightarrow+\infty,
\]
contradicting the bound for $\phi$. If $a(x)<0$, the reversed inequality gives divergence to $-\infty$. Hence $a(x)=0$ everywhere, and $W_T=\Phi^*\psi_\infty$.

The continuous functions $W_0+C_R\tau^2/2$ decrease to this continuous limit. A second application of Dini's theorem gives uniform convergence of $W_0$ and also the pointwise inequality
\[
 W_0+\frac{C_R}{2}\tau^2\ge\Phi^*\psi_\infty.
\]
Since $nt=n-n\tau$, we conclude that $U=W_0+nt$ converges uniformly to $\Phi^*(\psi_\infty+n)$ and satisfies
\[
 U\ge\Phi^*(\psi_\infty+n)-n\tau-\frac{C_R}{2}\tau^2,
\]
which proves \eqref{ii:eq:terminaloneside}.
\end{proof}

\begin{remark}\label{ii:rem:local-potential}
The estimates of this section hold on every relatively compact
regular product chart, with constants depending on a larger product
chart. In particular, the limiting base potential is
$C^{1,1}_{\mathrm{loc}}$ on $Y^\circ$, and the two convergences in
\eqref{ii:eq:terminal} are uniform on $\Phi^{-1}(K)$ for every
$K\Subset Y^\circ$. Indeed, the horizontal arguments use cutoffs
only in the base and the entire compact fibre. Fibrewise constancy
of the pointwise limit in this local setting follows directly from
\[
 \ddc_{F_q}\phi(t)\ge-\tau\omega_0|_{F_q},\qquad
 \ddc_{F_q}\phi_T\ge0.
\]
To identify the pointwise limit as a plurisubharmonic function,
fix $t_0<1$ and choose a local potential $h$ for $\omega_0|_{F_q}$.
For $t\ge t_0$, the functions
\[
 \phi(t)|_{F_q}-C_\phi t+(1-t_0)h
\]
are decreasing and plurisubharmonic. Their bounded limit is
plurisubharmonic by the decreasing-limit property
\cite[Chapter~I, Theorem~5.4]{Demailly}. Thus
$\phi_T|_{F_q}+(1-t_0)h$ is plurisubharmonic for every $t_0$.
Letting $t_0\uparrow1$ in the submean inequality proves that
$\phi_T|_{F_q}$ itself is plurisubharmonic. It is therefore
constant on the compact connected fibre $F_q$. The one-sided estimate is unchanged. Dini's
theorem is applied to the compact set $\Phi^{-1}(K)$, whose limiting
potential is continuous. This assertion requires no regularity of
the limiting potential along the excluded singular fibres.
\end{remark}

\section{Ricci vertices and the Li--Yau estimate}\label{ii:sec:vertices}

The freedom to change the representative of the limiting class
allows us to change the minimum of its Ricci potential. We choose
the representative so that the terminal potential has a quadratic
well at a prescribed base point. The minimum principle then places
a genuine Ricci vertex within distance $C\sqrt{1-t}$ of that
point. The Li--Yau estimate of \cite{JST} controls the resulting
potentials, while a local change of gauge removes the base term
from the Ricci identity.

We use the notation and estimates of Section~\ref{ii:sec:preliminaries},
with $T=1$, $\tau=e^{-s}=1-t$ and $\widetilde{\omega}=\tau^{-1}\omega$.
For a smooth real function $\rho$ on $Y$, set
\[
 u_\rho=e^s(U+\Phi^*\rho),\qquad \theta_{Y,\rho}=\omega_Y-\ddc\rho.
\]
The identity for $\ddc U$ in \eqref{ii:eq:Uident}, multiplied by $e^s$, gives
\begin{equation}\label{ii:eq:vertexidentity}
 \Ric(\widetilde{\omega})+\ddc u_\rho=\widetilde{\omega}-e^s\Phi^*\theta_{Y,\rho}.
\end{equation}
Here $\Ric(\widetilde{\omega})=\Ric(\omega)$ because the metrics differ by a spatially constant factor. Following \cite[Definition~1.2]{JST}, we call a global minimum of $u_\rho$ a Ricci vertex for the representative $\theta_{Y,\rho}$. The form $\theta_{Y,\rho}$ is smooth and closed and represents $[\omega_Y]$; it need not be positive. Since $e^s>0$, the minimum points of $u_\rho$ and $U+\Phi^*\rho$ coincide at each time.

\begin{lemma}\label{ii:lem:wells}
There are smooth real functions $\rho_y$ on $Y$, indexed by $y\in Y$, and constants $c,C>0$, such that
\begin{equation}\label{ii:eq:well}
 \psi_\infty(z)+\rho_y(z)\ge \psi_\infty(y)+\rho_y(y)+c\,d_{\omega_Y}(z,y)^2,
 \qquad \|\rho_y\|_{C^4(Y)}\le C.
\end{equation}
The functions may be chosen with $\rho_y(y)=0$. No regularity of the family in its parameter $y$ is required.
\end{lemma}
\begin{proof}
Choose a finite collection of coordinate balls and relatively compact inner balls covering $Y$. Shrink the inner balls so that a Euclidean ball of one fixed radius $r_*>0$ about each of their points stays inside the corresponding larger chart. All coordinate and metric comparison constants are uniform in this finite collection.

By Proposition~\ref{ii:prop:terminal}, $D\psi_\infty$ is Lipschitz in these coordinates with a common constant $L$. For $z$ in the coordinate ball about $y$, the fundamental theorem of calculus gives
\begin{align}
 \psi_\infty(z)-\psi_\infty(y)-D\psi_\infty(y)(z-y)
 &=\int_0^1\bigl(D\psi_\infty(y+a(z-y))-D\psi_\infty(y)\bigr)(z-y)\,\dd a\notag\\
 &\ge-\int_0^1 La|z-y|^2\,\dd a
 =-\frac L2|z-y|^2.\label{ii:eq:taylorlower}
\end{align}
Fix $A>L/2+1$ and put
\[
 P_y(z)=-D\psi_\infty(y)(z-y)+A|z-y|^2.
\]
Then $P_y(y)=0$ and
\begin{equation}\label{ii:eq:localwellgap}
 \psi_\infty(z)+P_y(z)-\psi_\infty(y)\ge |z-y|^2.
\end{equation}
Choose $\chi_y=1$ on $|z-y|\le r_*/2$, with support in $|z-y|<r_*$, $0\le\chi_y\le1$, and uniform derivatives through order four. Choose a fixed $M$ so large that
\[
 \psi_\infty(z)+M-\psi_\infty(y)\ge1\qquad(y,z\in Y).
\]
Define $\rho_y=\chi_yP_y+(1-\chi_y)M$ on the chart and extend it by $M$ outside. It is smooth and $\rho_y(y)=0$. On the support of $\chi_y$,
\begin{align*}
 \psi_\infty(z)+\rho_y(z)-\psi_\infty(y)
 &=\chi_y(z)\bigl(\psi_\infty(z)+P_y(z)-\psi_\infty(y)\bigr)\\
 &\quad +(1-\chi_y(z))\bigl(\psi_\infty(z)+M-\psi_\infty(y)\bigr)\\
 &\ge\chi_y(z)|z-y|^2+(1-\chi_y(z)).
\end{align*}
For $|z-y|\le r_*/2$, this is at least $|z-y|^2$. On the transition annulus it is at least $\min\{r_*^2/4,1\}$, and outside the support it is at least one. Coordinate comparison on the inner ball and the finite diameter of $(Y,\omega_Y)$ therefore give a single $c>0$ for \eqref{ii:eq:well}. The coefficients of $P_y$ are uniformly bounded, so the fixed-radius cutoff construction gives the $C^4$ bound. This construction uses only $D\psi_\infty(y)$ and its Lipschitz constant, not higher derivatives of $\psi_\infty$.
\end{proof}

\begin{proposition}\label{ii:prop:vertices}
For every $y\in Y$ and $t<1$, any global minimum $p_{y,t}$ of $U(t)+\Phi^*\rho_y$ satisfies
\begin{equation}\label{ii:eq:anchor}
 d_{\omega_Y}(\Phi(p_{y,t}),y)\le D_0\sqrt\tau,\qquad
 -C\tau\le U(p_{y,t},t)-\psi_\infty(\Phi(p_{y,t}))-n\le C\tau.
\end{equation}
The constants are uniform in $y,t$ and in the choice of a minimizer.
\end{proposition}
The distance in \eqref{ii:eq:anchor} is measured on the base.
Thus the Ricci vertex lies in the tube
$\Phi^{-1}(\overline{B_{\omega_Y}(y,D_0\sqrt\tau)})$ around the prescribed fibre
$F_y$. This is the localization needed to apply the estimates
of \cite{JST} in Section~\ref{ii:sec:scalar}.

\begin{proof}
Fix $y$ and write
\[
 \mathcal F_y(x,t)=U(x,t)+\rho_y(\Phi(x)),\qquad m_y(t)=\min_X\mathcal F_y(\cdot,t),
 \qquad E_y(z)=\psi_\infty(z)+n+\rho_y(z).
\]
From \eqref{ii:eq:Uident},
\begin{equation}\label{ii:eq:minimumheat}
 \Box_{\omega} \mathcal F_y
 =\tr_\omega\Phi^*(\omega_Y-\ddc\rho_y)\ge-C_1.
\end{equation}
Indeed, $\|\rho_y\|_{C^2}\le C$ gives $\omega_Y-\ddc\rho_y\ge-C'\omega_Y$, while the Schwarz bound \eqref{ii:eq:basic} gives $\tr_\omega\Phi^*\omega_Y\le C''$; take $C_1=C'C''$.

Apply the minimum monotonicity in \cite[Lemma~4.2(1)]{JST}
to $u_y=e^s\mathcal F_y$ and write
$a_y(s)=\min_Xu_y(\cdot,s)=e^sm_y(t)$. After enlarging the
uniform constant $C_1$, that result says that
$e^{-s}(a_y(s)-C_1)$ is nondecreasing. Thus, for $t<t'<1$,
\[
 m_y(t')-C_1(1-t')\ge m_y(t)-C_1(1-t),\qquad
 m_y(t')\ge m_y(t)-C_1(t'-t).
\]
This comparison does not differentiate $m_y$ or a minimizing point.
We spell out the passage to the limiting minimum. Put
\[
 \varepsilon(t')=\|U(t')-\Phi^*\psi_\infty-n\|_\infty.
\]
For every $x\in X$, the definitions give
\begin{align*}
 \mathcal F_y(x,t')-E_y(\Phi(x))
 &=U(x,t')+\rho_y(\Phi(x))
     -\psi_\infty(\Phi(x))-n-\rho_y(\Phi(x))\\
 &=U(x,t')-\psi_\infty(\Phi(x))-n.
\end{align*}
Consequently
\[
 E_y(\Phi(x))-\varepsilon(t')
 \le\mathcal F_y(x,t')
 \le E_y(\Phi(x))+\varepsilon(t').
\]
Let $e_y=\min_X\Phi^*E_y$. The left inequality holds at every
point, so taking the minimum gives $m_y(t')\ge e_y-\varepsilon(t')$.
For the reverse bound, choose a point $x_y^*$ minimizing
$\Phi^*E_y$, which exists by compactness. Evaluating the right
inequality at this point gives
\[
 m_y(t')\le\mathcal F_y(x_y^*,t')
 \le E_y(\Phi(x_y^*))+\varepsilon(t')
 =e_y+\varepsilon(t').
\]
Combining these two bounds and using the uniform convergence of $U$,
\[
 \left|m_y(t')-\min_X\Phi^*E_y\right|
 \le\|U(t')-\Phi^*\psi_\infty-n\|_\infty\longrightarrow0.
\]
The map $\Phi$ is surjective, so
$e_y=\min_YE_y$. Moreover, \eqref{ii:eq:well} says
$E_y(z)-E_y(y)\ge c\,d_{\omega_Y}(z,y)^2\ge0$; hence $e_y=E_y(y)$.
Rearranging the preceding heat comparison, for each fixed $t<t'<1$,
\[
 m_y(t)\le m_y(t')+C_1(t'-t).
\]
Now $m_y(t')\to E_y(y)$ and $t'-t\to1-t=\tau$. Passing to
the limit gives
\begin{equation}\label{ii:eq:minimumupper}
 m_y(t)\le E_y(y)+C_1\tau.
\end{equation}

Now fix any minimizer $p=p_{y,t}$ and put $z=\Phi(p)$. By \eqref{ii:eq:terminaloneside} and \eqref{ii:eq:well}, respectively,
\[
 U(p,t)\ge \bigl(\psi_\infty(z)+n\bigr)-n\tau-C_2\tau^2,
 \qquad E_y(z)\ge E_y(y)+c\,d_{\omega_Y}(z,y)^2.
\]
Adding $\rho_y(z)$ to the first inequality and then using the second gives the full comparison
\begin{align}
 E_y(y)+c\,d_{\omega_Y}(z,y)^2-n\tau-C_2\tau^2
 &\le E_y(z)-n\tau-C_2\tau^2\notag\\
 &\le U(p,t)+\rho_y(z)
 =m_y(t)\notag\\
 &\le E_y(y)+C_1\tau.\label{ii:eq:vertexchain}
\end{align}
After subtracting $E_y(y)$ and moving the time terms to the right,
\[
 c\,d_{\omega_Y}(z,y)^2\le(n+C_1)\tau+C_2\tau^2
 \le(n+C_1+C_2)\tau,
\]
since $0<\tau\le1$. Thus the first assertion holds with
$D_0=((n+C_1+C_2)/c)^{1/2}$.

For the height, the minimum identity and \eqref{ii:eq:minimumupper} give
\begin{align*}
 U(p,t)-\psi_\infty(z)-n
 &=m_y(t)-E_y(z)\\
 &\le E_y(y)+C_1\tau-E_y(z)\\
 &\le C_1\tau-c\,d_{\omega_Y}(z,y)^2\le C_1\tau.
\end{align*}
The opposite inequality is
\[
 U(p,t)-\psi_\infty(z)-n\ge-n\tau-C_2\tau^2\ge-(n+C_2)\tau.
\]
All constants above were fixed independently of the chosen minimizer.
\end{proof}

\begin{remark}\label{ii:rem:local-vertices}
By \eqref{ii:eq:well}, $\psi_\infty+\rho_y$ has its unique base
minimum at $y$, with value $\psi_\infty(y)$ since $\rho_y(y)=0$.
Its pullback is therefore minimized precisely on the fibre $F_y$.
\end{remark}

The next proposition is the Li--Yau estimate for the Ricci
potential associated to $\rho_y$. Jian--Song--Tian
\cite[Proposition~5.2]{JST} obtain gradient and Laplacian estimates
after subtracting a suitable time-dependent lower bound from the
potential. Here the unique base minimum identifies that lower
bound explicitly, so their estimate applies uniformly in $y$.
The time-derivative estimate then follows from the potential equation.

\begin{proposition}\label{ii:prop:differential}
There is a constant $C_0>0$, independent of $y$, such that
\begin{equation}\label{ii:eq:jst-normalization}
 v_y=e^s\bigl(U+\Phi^*\rho_y-\psi_\infty(y)-n\bigr)+C_0+1
\end{equation}
satisfies $v_y\ge1$ on $X\times[0,\infty)$ and
\begin{equation}\label{ii:eq:jst-differential}
 |\partial v_y|_{\widetilde{\omega}}^2+|\Delta_{\widetilde{\omega}} v_y|
       +|(v_y)_s|\le Cv_y,\qquad
 |\Box_{\widetilde{\omega}}v_y|\le Cv_y.
\end{equation}
At every vertex $p_{y,t}$ in Proposition~\ref{ii:prop:vertices},
$v_y(p_{y,t},s)\le C$. The constants are uniform in $y$ and $s$.
\end{proposition}
\begin{proof}
Use the potential $u_{\rho_y}=e^s(U+\Phi^*\rho_y)$ introduced at the beginning of this section. The one-sided terminal estimate and
\eqref{ii:eq:well}, with $\rho_y(y)=0$, give
\[
 u_{\rho_y}(x,s)
 \ge e^s\bigl(\psi_\infty(y)+n\bigr)-n-C_2.
\]
For the application of \cite[Proposition~5.2]{JST}, choose the
lower bound
\[
 b(s)=e^s\bigl(\psi_\infty(y)+n\bigr)-C_0,
 \qquad C_0\ge n+C_2.
\]
Then
\[
 b(s)\le\min_Xu_{\rho_y}(\cdot,s),\qquad
 b'=b+C_0,\qquad |b(s)|\le Ce^s,
\]
and $v_y=u_{\rho_y}-b+1\ge1$. These are the normalization
conditions in the cited proposition. The uniform $C^4$ bound for
$\rho_y$ and Remark~\ref{ii:rem:analytic-inputs} give uniform
constants. Applying that proposition on $[0,S]$, with constants
independent of $S$, yields
\[
 |\partial v_y|_{\widetilde{\omega}}^2+|\Delta_{\widetilde{\omega}} v_y|\le Cv_y.
\]

Here $b$ is spatially constant, so the spatial derivatives of
$v_y$ and $u_{\rho_y}$ agree.
The evolution and trace identities
\cite[equations~(4.15) and~(4.19)]{JST}, in our normalization, give
\[
 \begin{aligned}
 (u_{\rho_y})_s&=u_{\rho_y}+n-R_{\widetilde{\omega}},\\
 \Delta_{\widetilde{\omega}}u_{\rho_y}
 &=n-R_{\widetilde{\omega}}
       -e^s\tr_{\widetilde{\omega}}\Phi^*\theta_{Y,\rho_y}.
 \end{aligned}
\]
Consequently
\begin{align*}
 (v_y)_s&=v_y-1-C_0+n-R_{\widetilde{\omega}},\\
 \Box_{\widetilde{\omega}} v_y
 &=v_y-1-C_0+e^s\tr_{\widetilde{\omega}}\Phi^*\theta_{Y,\rho_y}.
\end{align*}
The trace term is uniformly bounded, since
$-C\omega_Y\le\theta_{Y,\rho_y}\le C\omega_Y$ and
$e^s\tr_{\widetilde{\omega}}\Phi^*\omega_Y=\tr_\omega\Phi^*\omega_Y\le C$.
The second identity therefore gives the heat-operator bound;
combining it with the Laplacian estimate gives the time-derivative
bound. Finally, \eqref{ii:eq:minimumupper} and
$\rho_y(y)=0$ imply at any vertex
\[
 v_y(p_{y,t},s)
 =e^s\bigl(m_y(t)-\psi_\infty(y)-n\bigr)+C_0+1
 \le C_1+C_0+1.
\]
\end{proof}

\begin{remark}\label{ii:rem:local-jst}
For the local assertion in Remark~\ref{ii:rem:local}, fix
$K\Subset Y^\circ$ and buffered regular product charts covering $K$.
Remark~\ref{ii:rem:local-potential} gives
$\psi_\infty\in C^{1,1}_{\mathrm{loc}}(Y^\circ)$ and uniform convergence
of $U$ to $\Phi^*\psi_\infty+n$ on the inverse image of every
compact subchart. In the construction of Lemma~\ref{ii:lem:wells},
take the support of $\chi_y$ inside a buffered chart and extend
$\rho_y$ by a constant $M$ on its complement. Choose $M$ to satisfy
the well requirement on that chart and $M>2C_U+1$, where
$C_U=\sup_{X\times[0,1)}|U|$. Since $\rho_y(y)=0$, outside the
chart and on a collar where $\rho_y=M$ one has
\[
 U(x,t)+M\ge M-C_U>C_U+1
       \ge\sup_{F_y}U(\cdot,t)+1.
\]
Every global minimum of $U+\Phi^*\rho_y$ therefore lies in a fixed
compact regular subchart. The local quadratic well inequality
holds there, and its unique base minimum is $y$. Uniform
convergence on this compact subchart gives
$m_y(t)\to\psi_\infty(y)+n$ by the two inequalities used in
Proposition~\ref{ii:prop:vertices}. The minimum comparison is
global: derivatives of $\rho_y$ are supported in the regular
chart, and their traces are bounded by the parabolic Schwarz
estimate. The one-sided estimate \eqref{ii:eq:terminaloneside}
is needed only at these minimizers. Consequently
\eqref{ii:eq:anchor} holds uniformly for $y\in K$, with constants
depending on the buffered charts. This argument requires no
extension of $\psi_\infty$ across the singular fibres.

The same global normalization applies to these localized wells.
Indeed, the minimum of $u_{\rho_y}=e^s(U+\Phi^*\rho_y)$ is attained
in the compact regular subchart. The one-sided bound and the
quadratic well inequality there give
\[
 \min_Xu_{\rho_y}(\cdot,s)
 \ge e^s\bigl(\psi_\infty(y)+n\bigr)-C_0=b(s)
\]
with a uniform constant. The one-sided bound applies for all
$s\ge0$, so this inequality holds on every full interval $[0,S]$.
The bounds for $b'$ and $e^{-s}|b|$ are unchanged. Each $\rho_y$
is constant outside a regular chart and has a smooth ambient
extension with uniformly bounded $C^4$ norm. Thus
\cite[Proposition~5.2]{JST} applies on the whole compact manifold
$X$, with constants uniform for $y\in K$, and gives the estimates
of Proposition~\ref{ii:prop:differential}.
\end{remark}

For the fibrewise estimates below, we also record a local potential whose Ricci identity has no horizontal term. Here $B_a(y)$ denotes a Euclidean coordinate ball in a fixed product chart.

\begin{lemma}\label{ii:lem:localgauge}
Fix $L\ge1$, a sufficiently late normalized time $s_0\ge3$, and $r=e^{-s_0/2}$, so that the closed ball below lies in a fixed product chart. Suppose a smooth real function $\ell$ on a neighborhood of $\overline{B_{8Lr}(y)}$ satisfies
\begin{equation}\label{ii:eq:ell}
 \ddc\ell=-\omega_Y,\qquad |\psi_\infty+n-\ell|\le C_Lr^2.
\end{equation}
Assume $U\ge\Phi^*\psi_\infty+n-C\tau$ on this tube. For suitable constants $C_0,d$, independent of $s_0,y$, the function
\begin{equation}\label{ii:eq:vdef}
 v=e^s(U-\Phi^*\ell)+C_0r^2e^s+d
\end{equation}
is at least one on $\Phi^{-1}(B_{8Lr}(y))\times[s_0-3,s_0+1]$ and satisfies
\begin{equation}\label{ii:eq:untwisted}
 \Ric(\widetilde{\omega})+\ddc v=\widetilde{\omega},\qquad
 \Box_{\widetilde{\omega}} v=v-d,\qquad
 v_s=v-d+n-R_{\widetilde{\omega}}.
\end{equation}
On this interval,
\begin{equation}\label{ii:eq:vcompare}
 \left|v-e^s(U-\Phi^*\psi_\infty-n)\right|\le C_L'.
\end{equation}
Such $\ell$ exists under Proposition~\ref{ii:prop:terminal}.
\end{lemma}
\begin{proof}
Choose a fixed smooth local potential $\psi$ with $\ddc\psi=\omega_Y$. Let $P_y$ be the real affine Taylor polynomial of $\psi_\infty+n+\psi$ at $y$, and put $\ell=P_y-\psi$. The function $P_y$ is pluriharmonic. The $C^{1,1}$ bound and the Taylor integral formula give
\[
 \ddc\ell=-\omega_Y,\quad
 |\psi_\infty(z)+n-\ell(z)|
 =|\psi_\infty(z)+n+\psi(z)-P_y(z)|\le C|z-y|^2\le C_Lr^2.
\]
The construction is valid on a neighborhood of the closed ball after first fixing a larger coordinate neighborhood.

The assumed one-sided bound gives
\[
 e^s(U-\Phi^*\ell)
 =e^s(U-\Phi^*\psi_\infty-n)+e^s\Phi^*(\psi_\infty+n-\ell)
 \ge-C-C_Lr^2e^s.
\]
Choose $C_0\ge C_L$ and $d\ge C+1$. Then \eqref{ii:eq:vdef} gives $v\ge1$ on the whole stated cylinder.

We check the equations, keeping $r,s_0,\ell$ fixed during differentiation. Since $\dd t/\dd s=\tau$ and $\Delta_{\widetilde{\omega}}=\tau\Delta_\omega$, \eqref{ii:eq:Uident} gives
\begin{align}
 \Ric(\widetilde{\omega})+\ddc u_0&=\widetilde{\omega}-e^s\Phi^*\omega_Y,\notag\\
 \Box_{\widetilde{\omega}} u_0
 &=u_0+e^s\tr_{\widetilde{\omega}}\Phi^*\omega_Y,
 &(u_0)_s&=u_0+n-R_{\widetilde{\omega}}.\label{ii:eq:uglobalnormalized}
\end{align}
Also
\[
 \Delta_{\widetilde{\omega}}\Phi^*\ell=-\tr_{\widetilde{\omega}}\Phi^*\omega_Y,
 \qquad
 \Box_{\widetilde{\omega}} (-e^s\Phi^*\ell)
 =-e^s\Phi^*\ell-e^s\tr_{\widetilde{\omega}}\Phi^*\omega_Y.
\]
The horizontal trace terms cancel. The term $C_0r^2e^s$ is spatially constant and has time derivative itself, while $d$ has zero derivatives. Consequently
\[
 \Box_{\widetilde{\omega}} v=u_0-e^s\Phi^*\ell+C_0r^2e^s=v-d.
\]
Similarly,
\[
 \Ric(\widetilde{\omega})+\ddc v
 =\widetilde{\omega}-e^s\Phi^*\omega_Y-e^s\Phi^*(\ddc\ell)=\widetilde{\omega}.
\]
Differentiating \eqref{ii:eq:vdef} in time and using \eqref{ii:eq:uglobalnormalized} proves the last equation in \eqref{ii:eq:untwisted}. Finally,
\[
 v-e^s(U-\Phi^*\psi_\infty-n)
 =e^s\Phi^*(\psi_\infty+n-\ell)+C_0r^2e^s+d,
\]
whose absolute value is at most $(C_L+C_0)e+d$ because $r^2e^s=e^{s-s_0}\le e$ on the cylinder. This proves \eqref{ii:eq:vcompare}.
\end{proof}

\section{Proof of Theorem 1.1}\label{ii:sec:scalar}

The estimates of \cite[Theorem~1.6]{JST} now apply, using the
tube-volume bound \eqref{ii:eq:tubevolume} from
Section~\ref{ii:sec:preliminaries} and the Ricci vertices constructed
in Section~\ref{ii:sec:vertices}. There is one
geometric point to check: the tube controlled by that theorem
must contain the entire prescribed fibre, even though the vertex
need not lie on it. A fixed enlargement of the auxiliary base
scale achieves this, and the horizontal volume estimate controls
the corresponding larger tube. This proves
Theorem~\ref{ii:thm:main} with constants independent of the fibre.

\begin{proof}[Proof of Theorem~\ref{ii:thm:main}]
First assume $T=1$. Fix $y\in Y$ and a sufficiently late time $t$.
Let $p=p_{y,t}$ be a vertex in
Proposition~\ref{ii:prop:vertices}, and write
\[
 z=\Phi(p),\qquad r=\sqrt{1-t}.
\]
Choose a fixed $\Lambda>D_0+1$, and use the auxiliary base metric
$g_Y=\Lambda^{-2}g_{\omega_Y}$. This metric specifies the base tubes;
it does not change the representative $\theta_{Y,\rho_y}$ defining
the vertex. Since $d_{g_Y}=\Lambda^{-1}d_{\omega_Y}$,
\[
 d_{g_Y}(z,y)\le\Lambda^{-1}D_0r<r,\qquad
 F_y\subset\Phi^{-1}(B_{g_Y}(z,r)).
\]
In particular, the unit-scale tube around the vertex contains the
entire fibre whose geometry is to be estimated.

The volume hypothesis of \cite[Theorem~1.6]{JST} concerns the
tenfold base tube at the same time. By \eqref{ii:eq:real} and
\eqref{ii:eq:tubevolume},
\begin{align}
 \Vol_{g(t)}\Phi^{-1}(B_{g_Y}(z,10r))
 &=2^n\Vol_{\omega(t)}\Phi^{-1}(B_{\omega_Y}(z,10\Lambda r))\notag\\
 &\le C(1-t)^{n-m}(10\Lambda r)^{2m}\notag\\
 &=C(10\Lambda)^{2m}(1-t)^n
 =A_\Lambda(1-t)^n.\label{ii:eq:jst-tubevolume}
\end{align}
The centre satisfies $d_{\omega_Y}(z,y)\le D_0r$. Thus the entire
tenfold ball lies in $B_{\omega_Y}(y,(10\Lambda+D_0)r)$. Choose $t$
so late that $(10\Lambda+D_0)r$ is smaller than the fixed buffer
width of the product charts. Taking $t$ still closer to one if
necessary also makes the smaller geodesic balls connected. The threshold
and all coordinate comparisons are uniform in a finite atlas. The constant $A_\Lambda$ is fixed independently of $y,t$.

The point $p$ is a Ricci vertex for $\theta_{Y,\rho_y}$ by
\eqref{ii:eq:vertexidentity}. The perturbations $\rho_y$ have uniform
$C^4$ bounds, so the dependence on them in
\cite[Theorem~1.6]{JST} is uniform; see the explicit dependence in
\cite[Proposition~5.2 and proof of Theorem~6.8]{JST} and
Remark~\ref{ii:rem:analytic-inputs}. Applying that theorem in the
metric conventions of \eqref{ii:eq:real} yields
\[
 \sup_{\Phi^{-1}(B_{g_Y}(z,r))}|R_{\omega(t)}|
       \le\frac C{1-t},\qquad
 \diam_{(X,g(t))}\Phi^{-1}(B_{g_Y}(z,r))
       \le C\sqrt{1-t}.
\]
Restriction to $F_y$ proves both assertions at late times.
The diameter in the cited theorem is the ambient diameter.
Since $y$ was arbitrary, the scalar bound is global. Smoothness
on the remaining compact time interval extends both estimates to
$0\le t<1$.

For general $T$, put $\widehat t=t/T$ and
$\widehat\omega(\widehat t)=T^{-1}\omega(T\widehat t)$.
The corresponding real metric is
$\widehat g(\widehat t)=T^{-1}g(t)$, and
\[
 R_{\widehat\omega}(\widehat t)=TR_\omega(t),\qquad
 d_{\widehat g(\widehat t)}=T^{-1/2}d_{g(t)},\qquad
 1-\widehat t=\frac{T-t}{T}.
\]
The estimates just proved become
\[
 T|R_\omega(t)|\le\frac{CT}{T-t},\qquad
 T^{-1/2}\diam_{(X,g(t))}F_y
      \le C\sqrt{\frac{T-t}{T}},
\]
which give \eqref{ii:eq:main}.
\end{proof}

\begin{remark}\label{ii:rem:jst-packing-credit}
The passage to diameter control is the Harnack and ball-packing
argument of \cite[Theorem~1.6 and proof of Theorem~6.8]{JST}.
The preceding computation verifies its tube-volume hypothesis;
their theorem then gives both estimates on the prescribed fibre.
\end{remark}

\begin{remark}\label{ii:rem:local-scalar}
For the local assertion in Remark~\ref{ii:rem:local}, fix
$K\Subset Y^\circ$ and buffered regular product charts covering
$K$. Remark~\ref{ii:rem:local-jst} places the required global
vertices in these charts and proves \eqref{ii:eq:anchor} with a
constant $D_0$ depending on $K$. Choose $\Lambda>D_0+1$ with this
constant. At sufficiently late times all tenfold tubes in
\eqref{ii:eq:jst-tubevolume} stay inside the buffered charts.
The compact-local horizontal estimate and the fibre-volume
identity give precisely the same volume calculation there.
The global smooth representatives have uniform ambient $C^4$
bounds, as in Remark~\ref{ii:rem:local-jst}. Thus
\cite[Theorem~1.6]{JST} applies to each of these tubes and gives
both estimates uniformly over $K$. Enlarging the constants on a
compact initial time interval completes the local assertion.
\end{remark}

\section{Relative potentials and fibre estimates}\label{ii:sec:relative}

The normalized Ricci potential contains a large term pulled back
from the base. After subtracting its terminal profile, the
remaining potential and its relevant derivatives are uniformly
bounded. The quotient determinant relates this relative potential
to the intrinsic Ricci potential of the fibre. This gives, in
particular, a uniform intrinsic Poincar\'e inequality, which will
prevent fibre volume from being lost in a tangent limit.

Throughout this section, normalize $T=1$ and put
$\tau=1-t=e^{-s}$ and $\widetilde{\omega}=\tau^{-1}\omega$.
Let $\psi_\infty\in C^{1,1}(Y)$ be the limiting base potential
from Section~\ref{ii:sec:terminal}, so that
$\phi(t)\to\Phi^*\psi_\infty$ uniformly. Using $U$ from
\eqref{ii:eq:Uintro}, set
\[
 u_0=e^sU,\qquad
 w=e^s(U-\Phi^*\psi_\infty-n),\qquad
 \theta_\infty=\omega_Y+\ddc\psi_\infty.
\]
Thus $w$ is the relative potential obtained by subtracting the
limiting base profile from $u_0$.

\begin{theorem}\label{ii:thm:relative}
Under the hypotheses of Theorem~\ref{ii:thm:main}, one has
\begin{equation}\label{ii:eq:relativeintro}
 c\omega_Y\le\theta_\infty\le C\omega_Y,\qquad
 |w|+|w_s|+|\partial w|_{\widetilde{\omega}}^2+|\Delta_{\widetilde{\omega}} w|\le C.
\end{equation}
The current $\theta_\infty$ has bounded measurable coefficients, and the Laplacian bound is understood weakly and almost everywhere. The exponentially normalized intrinsic Ricci potentials of $\widetilde{\omega}(s)|_{F_y}$ are uniformly bounded. Moreover, if
\[
 \dd\nu_{y,t}=\frac{\dd V_{\omega(t)|_{F_y}}}{\Vol_{\omega(t)|_{F_y}}F_y},
 \qquad \bar v_{y,t}=\int_{F_y}v\,\dd\nu_{y,t},
\]
then every smooth real function $v$ on $F_y$ satisfies
\begin{equation}\label{ii:eq:poincareintro}
 \int_{F_y}|v-\bar v_{y,t}|^2\dd\nu_{y,t}
 \le C(T-t)\int_{F_y}|\partial v|^2_{\omega(t)|_{F_y}}\dd\nu_{y,t}.
\end{equation}
If $n-m=1$, the intrinsic fibre diameter is also bounded above by $C\sqrt{T-t}$.
\end{theorem}

Write $\widetilde{\omega}_y=\widetilde{\omega}|_{F_y}$. The quotient metrics $H_t,H$, determinant function $D$, and fibre
volume $V_F$ are those of Section~\ref{ii:pre:quotient}.

\begin{proposition}\label{ii:rel:current}
Let $\psi_\infty$ be the limiting base potential. Then
\begin{equation}\label{ii:rel:currentbound}
 c\omega_Y\leq\theta_\infty:=\omega_Y+\ddc\psi_\infty\leq C\omega_Y
\end{equation}
as currents. In particular, $\theta_\infty$ has bounded measurable
coefficients.
\end{proposition}

This is the limiting-current estimate of Fu--Zhang
\cite[Theorem~1.1]{FuZhang}. Indeed, \eqref{ii:eq:terminal} identifies
its limiting current as
$\omega_T=\Phi^*(\omega_Y+\ddc\psi_\infty)=\Phi^*\theta_\infty$.
Their inequalities $c\Phi^*\omega_Y\le\omega_T\le C\Phi^*\omega_Y$
descend in local product charts to \eqref{ii:rel:currentbound}.

To relate $w$ to the smooth potential $v_y$ of
Proposition~\ref{ii:prop:differential}, observe that
\[
 v_y=w+C_0+1+e^s\Phi^*\bigl(\psi_\infty+\rho_y-\psi_\infty(y)\bigr).
\]
The last term is nonnegative and vanishes, together with its first
spatial derivatives, on $F_y$. Thus $v_y=w+C_0+1$ and
$\partial v_y=\partial w$ there, including in the horizontal
directions. This allows the Li--Yau estimate for $v_y$ to control
the gradient of $w$, although $\psi_\infty$ is only $C^{1,1}$.

\begin{proposition}\label{ii:rel:potential}
There is a constant $C>0$ such that
\begin{equation}\label{ii:rel:wbound}
 |w|+|w_s|+|\partial w|_{\widetilde{\omega}}^2+|\Delta_{\widetilde{\omega}} w|\leq C.
\end{equation}
The Laplacian identity and its bound hold weakly and almost everywhere.
\end{proposition}

\begin{proof}
Theorem~\ref{ii:thm:main} and $R_{\widetilde{\omega}}=\tau R_\omega$ give
$|R_{\widetilde{\omega}}|\leq C$. Hence $U_t=n-R_{\widetilde{\omega}}$ implies $|U_t|\leq C$.
For $t<t'<1$, integration at a fixed point $x\in X$ gives
\[
 |U(x,t')-U(x,t)|\leq C(t'-t).
\]
Letting $t'\uparrow1$ and using \eqref{ii:eq:terminal}, we obtain
\begin{equation}\label{ii:rel:terminalrate}
 |U-\Phi^*\psi_\infty-n|\leq C\tau,\qquad |w|\leq C.
\end{equation}
Since $\dd t/\dd s=\tau$, direct differentiation gives
\[
 w_s=e^s(U-\Phi^*\psi_\infty-n)+e^s\tau U_t=w+U_t.
\]
Also \eqref{ii:eq:Uident} and the invariance of the Ricci form under
constant rescaling give
\[
 \begin{split}
 \ddc w
 &=e^s\ddc U-e^s\Phi^*\ddc\psi_\infty\\
 &=\widetilde{\omega}-\Ric(\widetilde{\omega})-e^s\Phi^*(\omega_Y+\ddc\psi_\infty).
 \end{split}
\]
Consequently
\begin{equation}\label{ii:rel:wequations}
 w_s=w+n-R_{\widetilde{\omega}},\qquad
 \Delta_{\widetilde{\omega}} w=n-R_{\widetilde{\omega}}-e^s\tr_{\widetilde{\omega}}\Phi^*\theta_\infty.
\end{equation}
The first identity proves $|w_s|\leq C$. For the second, the pullback
of $\psi_\infty$ is defined as the function $\psi_\infty\circ\Phi$; it is
$C^{1,1}$ in every product chart. Thus the displayed second-derivative
identities are valid in distributions and almost everywhere. Moreover,
\[
 0\leq e^s\tr_{\widetilde{\omega}}\Phi^*\theta_\infty
 \leq Ce^s\tr_{\widetilde{\omega}}\Phi^*\omega_Y
 =C\tr_\omega\Phi^*\omega_Y\leq C
\]
almost everywhere, by \eqref{ii:rel:currentbound} and the Schwarz
bound \eqref{ii:eq:SQ} from Section~\ref{ii:sec:preliminaries}. This proves the
Laplacian bound.

For the gradient bound, fix $y\in Y$ and use $v_y$ from
Proposition~\ref{ii:prop:differential}, which satisfies
\[
 |\partial v_y|_{\widetilde{\omega}}^2\le Cv_y.
\]
Since $\psi_\infty+\rho_y$ has its minimum at $y$ and
$\rho_y(y)=0$, on the entire fibre $F_y$ we have
\[
 v_y=w+C_0+1,\qquad
 \partial v_y-\partial w
   =e^s\Phi^*\partial(\psi_\infty+\rho_y)=0.
\]
The second equality is an equality of ambient covectors at points
of $F_y$, including their horizontal components. By
\eqref{ii:rel:terminalrate}, $v_y\le C$ on $F_y$. Consequently
\[
 |\partial w|_{\widetilde{\omega}}^2\big|_{F_y}
 =|\partial v_y|_{\widetilde{\omega}}^2\big|_{F_y}\le C.
\]
All constants are independent of $y$ and $s$, which completes the
proof.
\end{proof}

\begin{remark}\label{ii:rem:jst-relative-credit}
The relative gradient estimate thus follows from
\cite[Proposition~5.2]{JST}: the function $\psi_\infty+\rho_y$
has vanishing differential at the chosen base point, so the ambient
derivatives agree on its fibre.
\end{remark}

Let $f_y(s)$ be the normalized intrinsic Ricci potential of $\widetilde{\omega}_y$:
\begin{equation}\label{ii:rel:intrinsic}
 \Ric(\widetilde{\omega}_y)+\ddc_{F_y}f_y=\widetilde{\omega}_y,
 \qquad \int_{F_y}e^{-f_y}\,\dd V_{\widetilde{\omega}_y}=V_F.
\end{equation}
The next computation constructs this potential explicitly, and so also
verifies its existence and its normalization.

\begin{proposition}\label{ii:rel:fbound}
The potentials in \eqref{ii:rel:intrinsic} satisfy $|f_y|\leq C$, uniformly
in $y$ and $s$.
\end{proposition}

\begin{proof}
The block determinant identity \eqref{ii:lim:eq:schur}, applied to
$\widetilde{\omega}$, gives $\det\widetilde{\omega}=(\det\widetilde{\omega}_y)(\det H)$.
Here the determinants are taken in the corresponding Hermitian coefficient matrices. Since
$\log\det H=D+ms+\log\det\omega_Y$ and the last two terms are constant
on each fibre, taking $-\ddc_{F_y}$ gives
\[
 \Ric(\widetilde{\omega})|_{F_y}=\Ric(\widetilde{\omega}_y)-\ddc_{F_y}D.
\]
Let $\iota_y:F_y\hookrightarrow X$ be the inclusion. Since
$\Phi\circ\iota_y$ is constant,
$\iota_y^*\Phi^*\omega_Y=0$ and
$\iota_y^*(\ddc u_0)=\ddc_{F_y}(u_0|_{F_y})$.
Restriction of \eqref{ii:eq:vertexidentity} with $\rho=0$ therefore gives
\[
 \Ric(\widetilde{\omega})|_{F_y}+\ddc_{F_y}u_0=\widetilde{\omega}_y.
\]
Substituting the determinant identity from the preceding line yields
\[
 \widetilde{\omega}_y
 =\Ric(\widetilde{\omega}_y)-\ddc_{F_y}D+\ddc_{F_y}u_0
 =\Ric(\widetilde{\omega}_y)+\ddc_{F_y}(u_0-D).
\]
Define
\[
 I(y,s)=\frac1{V_F}\int_{F_y}e^{D-u_0}\,\dd V_{\widetilde{\omega}_y},
 \qquad f_y=u_0-D+\log I(y,s).
\]
The function $I(y,s)$ is positive and constant along a fibre, and
\[
 \int_{F_y}e^{-f_y}\,\dd V_{\widetilde{\omega}_y}
 =I(y,s)^{-1}\int_{F_y}e^{D-u_0}\,\dd V_{\widetilde{\omega}_y}=V_F.
\]
This proves \eqref{ii:rel:intrinsic}. Two solutions differ by a real
pluriharmonic function, hence by a constant on the compact connected
fibre; the integral normalization fixes that constant.

For the uniform bound it is useful to cancel the large base term before
taking logarithms. Since $u_0=e^s\bigl(\psi_\infty(y)+n\bigr)+w$ on $F_y$, set
\[
 J(y,s)=\frac1{V_F}\int_{F_y}e^{D-w}\,\dd V_{\widetilde{\omega}_y}.
\]
The base term is constant on $F_y$, so it factors out of the
integral:
\[
 I=e^{-e^s\bigl(\psi_\infty(y)+n\bigr)}J,\qquad
 \log I=-e^s\bigl(\psi_\infty(y)+n\bigr)+\log J.
\]
Thus the two possibly large base terms cancel exactly:
\[
 \begin{aligned}
 f_y&=\bigl(e^s(\psi_\infty(y)+n)+w\bigr)-D
       -e^s(\psi_\infty(y)+n)+\log J\\
    &=w-D+\log J.
 \end{aligned}
\]
Let $M=\sup_{y,s}\|D-w\|_{L^\infty(F_y)}<\infty$.
Since $\Vol_{\widetilde{\omega}_y}F_y=V_F$, integration of
$e^{-M}\le e^{D-w}\le e^M$ gives
\[
 e^{-M}\le J\le e^M,\qquad |\log J|\le M,
 \qquad |f_y|\le |w-D|+|\log J|\le2M.
\]
This argument requires no derivative estimate for $D$.
\end{proof}

\begin{lemma}[{\cite[Lemma~3.1]{CollinsSzekelyhidi}, $\alpha=0$}]
\label{ii:rel:weightedpoincare}
Let $(F,h)$ be a compact connected K\"ahler manifold and let
$\Ric(h)+\ddc f=h$ for a smooth real function $f$. For
$\dd\mu=e^{-f}\dd V_h/\int_F e^{-f}\dd V_h$, every smooth real
function $a$ satisfies
\begin{equation}\label{ii:rel:weightedinequality}
 \int_F\left|a-\int_Fa\,\dd\mu\right|^2\dd\mu
 \le\int_F|\bar\partial a|_h^2\dd\mu.
\end{equation}
\end{lemma}

To match the normalization in the cited statement, replace $f$ by
\[
 \widetilde f=f+\log\left(
          \frac{\int_F e^{-f}\dd V_h}{\Vol_h(F)}\right).
\]
Then $\int_F e^{-\widetilde f}\dd V_h=\Vol_h(F)$, while
$\Ric(h)+\ddc\widetilde f=h$ and the probability measure $\mu$
are unchanged. Also $|\bar\partial a|_h=|\partial a|_h$ for real $a$.

\begin{proof}[Completion of the proof of Theorem~\ref{ii:thm:relative}]
Proposition~\ref{ii:prop:terminal} gives the uniform convergence and
$C^{1,1}$ regularity. Propositions~\ref{ii:rel:current},
\ref{ii:rel:potential}, and \ref{ii:rel:fbound} prove the current and potential
assertions.

For the Poincar\'e estimate, set
$\dd\nu_y=V_F^{-1}\dd V_{\widetilde{\omega}_y}$ and
$\dd\mu_y=e^{-f_y}\dd\nu_y$. Both measures have mass one. If $a$ is
a smooth real function on $F_y$, put
$\bar a_y=\int a\,\dd\nu_y$ and $c_y=\int a\,\dd\mu_y$.
The identity
\[
 \int|a-c|^2\dd\nu_y
 =\int|a-\bar a_y|^2\dd\nu_y+|c-\bar a_y|^2
\]
shows that the unweighted mean minimizes the left side over constants.
Using this fact, then Lemma~\ref{ii:rel:weightedpoincare}, we obtain
\begin{align*}
 \int_{F_y}|a-\bar a_y|^2\dd\nu_y
 &\leq\int_{F_y}|a-c_y|^2\dd\nu_y\\
 &\leq e^{\sup f_y}\int_{F_y}|a-c_y|^2\dd\mu_y\\
 &\leq e^{\sup f_y}\int_{F_y}|\bar\partial a|_{\widetilde{\omega}_y}^2\dd\mu_y\\
 &\leq e^{\sup f_y-\inf f_y}
       \int_{F_y}|\partial a|_{\widetilde{\omega}_y}^2\dd\nu_y.
\end{align*}
In the last line we used $|\bar\partial a|=|\partial a|$ for real $a$.
Now
\[
 \omega(t)|_{F_y}=\tau\widetilde{\omega}_y,\qquad
 |\partial a|_{\widetilde{\omega}_y}^2
   =\tau|\partial a|_{\omega(t)|_{F_y}}^2,\qquad
 \frac{\dd V_{\omega(t)|_{F_y}}}
      {\Vol_{\omega(t)|_{F_y}}F_y}
   =\frac{\tau^{n-m}\dd V_{\widetilde{\omega}_y}}{\tau^{n-m}V_F}=\dd\nu_y.
\]
The uniform bound for $f_y$ therefore proves
\eqref{ii:eq:poincareintro} when $T=1$.

If $n-m=1$, the smooth Fano fibre is $\mathbb P^1$.
The intrinsic diameter bound follows from
\cite[Theorem~1.2, equation~(1.10)]{Splitting}.

To restore general $T$, apply the preceding estimates to
$\omega^{(1)}(a)=T^{-1}\omega(Ta)$ for $0\leq a<1$.
Then $\tau=1-a=(T-t)/T$ and
$|\partial a_0|_{\omega^{(1)}|_{F_y}}^2
=T|\partial a_0|_{\omega|_{F_y}}^2$ for a test function $a_0$.
Thus the Poincar\'e factor becomes $\tau T=T-t$, while lengths gain
the factor $\sqrt T$, giving $\sqrt T\sqrt\tau=\sqrt{T-t}$.
\end{proof}

\begin{remark}\label{ii:rem:local-relative}
Theorem~\ref{ii:thm:relative} also holds compact-locally over
$Y^\circ$. Indeed, its convergence and $C^{1,1}$ assertions follow
from Remark~\ref{ii:rem:local-potential}. The current comparison
is local in product coordinates, and the scalar and potential
estimates follow from Remark~\ref{ii:rem:local-scalar}. The
relative gradient argument uses the wells of Lemma~\ref{ii:lem:wells}
with the local extension and uniform application of
\cite[Proposition~5.2]{JST} in Remark~\ref{ii:rem:local-jst}.
The equality of ambient covectors on the central fibre is unchanged. Finally,
the determinant calculation and the weighted Poincar\'e
inequality take place on the entire compact fibre $F_y$, with
uniform bounds for $y$ in any fixed compact subset of $Y^\circ$.
Thus all constants in the theorem may be chosen uniformly on
such a compact subset.
\end{remark}

\section{Limiting geometry and entropy}\label{ii:lim:sec:geometry}

The splitting theorem of \cite{Splitting} leaves a residual
factor whose relation to the original fibre must still be
understood. The ambient diameter bound makes this factor compact;
the intrinsic Poincar\'e inequality shows that it retains the
entire fibre volume. These facts allow us to pass from convergence
on the regular set to convergence of complete anticanonical
section spaces, and hence to a $\Q$-Gorenstein degeneration of
the original fibre. The conclusions are collected in
Proposition~\ref{ii:lim:prop:soliton-static}. An entropy argument
at the end of the section gives the additional uniformity needed
when the base point moves with the scale.

We use the original time variable, with maximal time $T$, and put
$\tau=T-t$, $s=-\log\tau$, and $\widetilde{\omega}=\tau^{-1}\omega(t)$.
Throughout this section we assume only the hypotheses of
Theorem~\ref{ii:thm:main}; the moving-limit criterion explicitly
assumes constancy of the fixed-fibre terminal entropy.

\subsection{Fibre estimates and limiting measures}\label{ii:lim:subsec:geometry}

The collapse of a fibre determines a single terminal conjugate
heat flow, independent of the points chosen on that fibre. Its
dependence on the base is controlled by the horizontal metric
estimate. For the associated split tangents, the principal issue
is preservation of fibre volume. Applying the intrinsic Poincar\'e
inequality to cutoffs on the regular part resolves this issue and
also shows that the horizontal determinant becomes asymptotically
constant in fibrewise $L^2$.

We use the fibre metrics $h_q$, probability measures $\nu_{q,s}$,
volumes $V_F,V$, and quotient notation $H_t$ and $D_q$
from Section~\ref{ii:pre:quotient}. In particular, the determinant
identity is \eqref{ii:lim:eq:schur}, and the complex and real
fibre volumes are given by \eqref{ii:lim:eq:volumes}.

We recall the estimates of
Theorem~\ref{ii:thm:main}, Theorem~\ref{ii:thm:relative},
and the basic bounds in Section~\ref{ii:sec:preliminaries}.
For $t$ sufficiently close to $T$, there are constants independent
of $q$ and $t$ such that
\begin{align}
 c\omega_Y\le H_t\le C\omega_Y,\qquad
 &\omega(t)\ge c\Phi^*\omega_Y,\label{ii:lim:eq:SQ}\\
 \tau|R_{\omega(t)}|\le C,\qquad
 &\diam_{(X,g(t))}F_q\le C\sqrt\tau,\label{ii:lim:eq:scalar-diam}\\
 \Ric(h_q)+\ddc f_q=h_q,\qquad
 &\int_{F_q}e^{-f_q}\,d\nu_{q,s}=1,\qquad |f_q|\le C,\label{ii:lim:eq:fibre-f}\\
 \Var_{\nu_{q,s}}u
 &\le C\int_{F_q}|\partial u|_{h_q}^2\,d\nu_{q,s}.
 \label{ii:lim:eq:fibre-poincare}
\end{align}
Enlarging the constants on a compact earlier time interval gives
these bounds for every $0\le t<T$.
The quotient bounds are recalled from \eqref{ii:eq:SQ} in
Section~\ref{ii:pre:quotient}. The upper bound is used in the
heat-flow comparison, in bounding $D_q$, and in the section-extension norms.
The diameter in \eqref{ii:lim:eq:scalar-diam} uses ambient distances.
The Poincar\'e inequality \eqref{ii:lim:eq:fibre-poincare} uses the
intrinsic fibre gradient and ordinary normalized fibre volume.
This distinction is used in Proposition~\ref{ii:lim:prop:no-loss}.

The local Ricci-potential construction also gives the following estimate.
Choose holomorphic base coordinates
$w=(w^1,\ldots,w^m)$. On a tube over a fixed-radius ball in the magnified
base coordinates
\begin{equation}\label{ii:lim:eq:magnified}
 z=\tau^{-1/2}(w-w(q)),\qquad \widetilde{\omega}=\tau^{-1}\omega(t),
\end{equation}
there is a smooth function $v$ with
\begin{equation}\label{ii:lim:eq:tube-v}
 |v|\le C,\qquad \Ric(\widetilde{\omega})+\ddc v=\widetilde{\omega}.
\end{equation}
The constants are uniform after choosing fixed nested tube radii.
Indeed, the local Ricci potential constructed in
Lemma~\ref{ii:lem:localgauge} differs by a
uniformly bounded function from the normalized relative Ricci
potential. The latter is uniformly bounded on the whole tube by
Proposition~\ref{ii:rel:potential}. These estimates give \eqref{ii:lim:eq:tube-v} on the whole tube,
uniformly for centres in a compact regular product region.

For the real flow $g=2g_\omega$, the Sobolev inequality
\cite[Theorem~D$^*$, equation~(1.15)]{Ye} gives constants $A,B$
independent of $t<T$ such that, with $\chi=n/(n-1)$,
\[
 \left(\int_X|u|^{2\chi}dV_g\right)^{1/\chi}
 \le A\int_X\left(|\nabla u|_g^2+\tfrac14R_gu^2\right)dV_g
       +B\int_Xu^2dV_g.
\]
For $\widetilde g=\tau^{-1}g=2g_{\widetilde{\omega}}$,
\[
 dV_{\widetilde g}=\tau^{-n}dV_g,\qquad
 |\nabla u|_{\widetilde g}^2=\tau|\nabla u|_g^2,
 \qquad R_{\widetilde g}=\tau R_g.
\]
Multiplication of the preceding inequality by $\tau^{-(n-1)}$
therefore replaces $g$ by $\widetilde g$ and $B$ by $B\tau$.
Using $\tau\le T$, the scalar bound in
\eqref{ii:lim:eq:scalar-diam}, and
$dV_{\widetilde g}=2^n\dV_{\widetilde{\omega}}$,
$|\nabla u|_{\widetilde g}^2=|\partial u|_{\widetilde{\omega}}^2$, we obtain
\begin{equation}\label{ii:lim:eq:ambient-sobolev}
 \left(\int_X|u|^{2\chi}\dV_{\widetilde{\omega}}\right)^{1/\chi}
 \le C\int_X\bigl(|\partial u|_{\widetilde{\omega}}^2+|u|^2\bigr)\dV_{\widetilde{\omega}},
 \qquad \chi=\frac{n}{n-1}>1.
\end{equation}
For the local assertion, extend a function supported in a regular
tube by zero; only the scalar bound on its support is used.
Pulled-back cutoffs between fixed magnified base radii have bounded
gradients by the Schwarz bound \eqref{ii:eq:basic}. This supplies the ambient Sobolev and cutoff
bounds used in the section-extension argument.

Since $dt/ds=\tau$, the determinant identity \eqref{ii:lim:eq:schur} gives
\[
 \Ric(\omega(t))|_{F_q}=\Ric(h_q)-\ddc_FD_q,\qquad
 \partial_s h_q=h_q+\partial_t\omega(t)|_{F_q}.
\]
Consequently, the restricted normalized flow is
\begin{equation}\label{ii:lim:eq:forced-flow}
 \partial_s h_q=h_q-\Ric(h_q)+\ddc_FD_q
                   =\ddc_F(f_q+D_q).
\end{equation}
Thus, if $h_q=h_0+\ddc\varphi$, then
$\dot\varphi=f_q+D_q+c(s)$ for a function of time $c$.

We associate a limiting conjugate heat flow to each fibre. The
argument also proves local uniformity over the regular product locus
without imposing an upper metric bound near singular fibres.
Write $d_t=d_{g(t)}$.

\begin{proposition}\label{ii:lim:prop:heat}
For each $q\in Y$, all point conjugate heat kernels with poles
$(x_j,t_j)$, $x_j\in F_q$ and $t_j\nearrow T$, converge on compact
earlier time intervals to the same conjugate heat flow $\mu_t^q$.
There are uniform constants $C,L$ such that
\begin{align}
 \int_Xd_t(x,y)^2\,d\mu_t^q(y)&\le C(T-t)
       &&(x\in F_q),\label{ii:lim:eq:moment}\\
 d_{W_1}^{g(t)}(\mu_t^q,\mu_t^{q'})&\le Ld_{\omega_Y}(q,q').
       \label{ii:lim:eq:W1base}
\end{align}
At each fixed $t<T$, the density of $\mu_t^q$ depends continuously
on $q$ in $C^\infty(X)$. For a fibration with a regular product
locus, these assertions hold locally there, with constants uniform
on compact subsets; in \eqref{ii:lim:eq:W1base} the two points are taken
in a common sufficiently small coordinate ball.
\end{proposition}

\begin{proof}
Fix nested balls $B_0\Subset B_1\Subset B_2$ in a regular product
chart, with $q\in B_0$. Splicing $|w-w(q)|^2$ to a fixed positive
constant gives smooth nonnegative functions $\rho_q$, constant
outside $B_2$, such that
\[
 \rho_q=|w-w(q)|^2\ \text{on }B_1,\qquad
 \inf_{Y\setminus B_1}\rho_q\ge c_0>0,\qquad
 \|\rho_q\|_{C^2(B_2)}\le C_0.
\]
The constants are uniform because $B_0\Subset B_1$. On a singular
base, constancy off the regular chart makes $\Phi^*\rho_q$ smooth
on $X$, with derivatives supported in a compact regular region.
Holomorphicity and the Schwarz bound \eqref{ii:eq:basic} give $|\Delta_{g(t)}\Phi^*\rho_q|\le C_1$.
For $\nu_{j,t}=\nu_{x_j,t_j;t}$, heat duality therefore gives
\[
 \int_X\Phi^*\rho_q\,d\nu_{j,t}
 =\int_t^{t_j}\!\int_X\Delta_{g(a)}\Phi^*\rho_q\,d\nu_{j,a}\,da
 \le C_1(t_j-t).
\]
Backward heat smoothing gives subsequential convergence on compact
earlier time intervals to smooth positive probability densities.
Every limiting flow $\mu$ satisfies
\begin{equation}\label{ii:lim:eq:basemoment}
 \int_X\Phi^*\rho_q\,d\mu_t\le C_1(T-t).
\end{equation}
Passing the variance bound \cite[Corollary~3.8]{BamlerHeat} to the
limit at fixed $t<T$ also gives
\[
 \iint d_t(p,y)^2\,d\mu_t(p)d\mu_t(y)\le H_{2n}(T-t).
\]
Averaging the sum of these integrands yields, for $\tau=T-t$, a
point $p_t$ with
\[
 \rho_q(\Phi(p_t))+\int_Xd_t(p_t,y)^2\,d\mu_t(y)\le C\tau.
\]
For small $\tau$, this forces $\Phi(p_t)\in B_1$ and
$|w(\Phi(p_t))-w(q)|\le C\sqrt\tau$. The connecting straight
segment lies in $B_1$. By the quotient upper bound \eqref{ii:eq:SQ}, its horizontal lift from any
$x\in F_q$ has length at most $L|w(\Phi(p_t))-w(q)|$; the lift
exists throughout because it stays over a compact regular region.
Its endpoint and $p_t$ lie in the same fibre, so the ambient
fibre-diameter bound gives
\begin{equation}\label{ii:lim:eq:distance-est}
 d_t(x,p_t)\le L|w(\Phi(p_t))-w(q)|+D\sqrt\tau\le C\sqrt\tau.
\end{equation}
Hence
\[
 \int_Xd_t(x,y)^2\,d\mu_t(y)
 \le2d_t(x,p_t)^2+2\int_Xd_t(p_t,y)^2\,d\mu_t(y)\le C\tau.
\]
This proves \eqref{ii:lim:eq:moment} for the whole measure, including
its mass outside the chart. Smoothness and the diameter bound on a
compact earlier time interval cover the remaining times. In particular,
$d_{W_1}^{g(t)}(\delta_x,\mu_t)\le C\sqrt{T-t}$.

For two subsequential limits $\mu,\widetilde\mu$, backward
contraction \cite[Proposition~4.17]{BamlerStructure} and this bound
give, whenever $t<a<T$ and $x\in F_q$,
\[
 d_{W_1}^{g(t)}(\mu_t,\widetilde\mu_t)
 \le d_{W_1}^{g(a)}(\mu_a,\delta_x)
       +d_{W_1}^{g(a)}(\delta_x,\widetilde\mu_a)
 \le2C\sqrt{T-a}.
\]
Letting $a\nearrow T$ proves uniqueness, hence full convergence;
write the limit as $\mu^q$.

For nearby $q,q'\in B_0$, the same horizontal lift and fibre bound
give $d_a(x,x')\le L|w(q)-w(q')|+D\sqrt{T-a}$ for
$x\in F_q$, $x'\in F_{q'}$. Choose $x\in F_q$ and $x'\in F_{q'}$. By \eqref{ii:eq:SQ} and the fibre-diameter
bound, $d_a(x,x')\le L|w(q)-w(q')|+D\sqrt{T-a}$.
Backward contraction, the triangle inequality, and Cauchy--Schwarz give
\[
\begin{aligned}
d_{W_1}^{g(t)}(\mu_t^q,\mu_t^{q'})
&\le d_{W_1}^{g(a)}(\mu_a^q,\mu_a^{q'}) \le d_{W_1}^{g(a)}(\mu_a^q,\delta_x)+d_a(x,x')
   +d_{W_1}^{g(a)}(\delta_{x'},\mu_a^{q'})\\
&\le \left(\int_X d_a(x,y)^2\,d\mu_a^q(y)\right)^{1/2}
   +d_a(x,x') +\left(\int_X d_a(x',y)^2\,d\mu_a^{q'}(y)\right)^{1/2}\\
&\le L|w(q)-w(q')|+(2C+D)\sqrt{T-a}.
\end{aligned}
\]
Letting $a\nearrow T$ proves the claim.
Let $a\nearrow T$ and compare coordinate and base distances.
For a global bundle, a finite atlas and subdivision of a minimizing
base geodesic prove \eqref{ii:lim:eq:W1base} for every pair.

Finally, write $d\mu_t^q=v_t^q\dV_{g(t)}$. For fixed $t<a<T$,
heat duality gives
\[
 v_t^q(x)=\int_XK(y,a;x,t)\,d\mu_a^q(y).
\]
Every $x$-derivative of the smooth kernel is uniformly Lipschitz in
$y$ on the compact product. Thus, for every integer $b\ge0$,
\[
 \|v_t^q-v_t^{q'}\|_{C^b(X)}
 \le C_{b,t,a}d_{W_1}^{g(a)}(\mu_a^q,\mu_a^{q'})
 \le C_{b,t,a}L d_{\omega_Y}(q,q').
\]
All chart constants are uniform on compact regular subsets, proving
the stated continuity and the local version.
\end{proof}

Equation~\eqref{ii:lim:eq:moment} implies
$\Var_{g(t)}(\mu_t^q)\le4C(T-t)$, so $\mu^q$ is a limiting point
in the first-singular-time construction of
\cite[Definition~2.32]{BamlerStructure}.
Its tangents exist along every sequence of scales after taking a
subsequence, and are shrinking solitons by
\cite[Theorem~2.37]{BamlerStructure}.

Fix $q$. For $\tau_i\searrow0$ use
\begin{equation}\label{ii:lim:eq:rescaling}
 g_i(a)=\tau_i^{-1}g(T+\tau_i a),\qquad a<0,
 \quad \mu_{i,a}=\mu_{T+\tau_i a}^q.
\end{equation}
Write $g_S$ for the residual time-$-1$ metric below and $\omega_S$
for its complex form, so $g_S=2g_{\omega_S}$.

By \cite[Proposition~2.3, Corollary~3.1 and Theorem~3.3]{Splitting},
a tangent flow at a fixed limiting point above $q$ has
time-$-1$ slice
\begin{equation}\label{ii:lim:eq:split}
 Z=\C^m\times S
\end{equation}
as a K\"ahler metric and normal complex space. The residual factor
$S$ is connected, complete, locally compact and klt. Its regular
locus is dense, and its distance is the completion of the regular
length metric. The soliton potential is locally Lipschitz on $S$.
Regular convergence includes smooth convergence of the metrics
and complex structures. By
\cite[Proposition~3.7]{Splitting}, after a constant complex-linear
change of base coordinates, the normalized base maps converge
smoothly on regular charts to $\operatorname{pr}_{\C^m}$.

By \cite[Theorems~2.5 and~2.14]{BamlerStructure}, the
convergence is smooth on the full regular spacetime. Thus the
smooth-convergence locus in \cite[Theorem~9.31]{BamlerCompactness}
equals the regular locus.
Fix the regular convergence embeddings on an exhaustion, as in
\cite[Theorem~9.31]{BamlerCompactness}. For
$U\Subset U'\Subset\Reg S$, the pulled-back normalized base maps
on $B_\eps(0)\times U'$ converge smoothly to $(z,y)\mapsto z$.
For all large $i$, their zero sets near
$\{0\}\times\overline U$ are unique graphs $z=z_i(y)$ with
$z_i\to0$ smoothly. Their images in $X$ lie in $F_q$; we call
these images \emph{fibre graphs}. They are embedded and
single-sheeted, and their induced metrics and complex structures
converge smoothly. Using the same convergence embedding on nested
domains makes the graphs compatible by the implicit function theorem.

The curvature-radius estimates and singular cutoffs in
\cite[Proposition~3.2 and Lemma~3.1]{Splitting} imply zero weighted two-capacity of the
singular set. Once $S$ is compact, they give cutoffs
\begin{equation}\label{ii:lim:eq:capacity}
 \chi_j\in C_c^\infty(\Reg S),\quad 0\le\chi_j\le1,\quad
 \chi_j\to1\text{ locally on }\Reg S,\quad
 \int_{\Reg S}|\nabla\chi_j|^2\dV_S\longrightarrow0.
\end{equation}
We may choose $\chi_j=1$ on an increasing compact exhaustion of
$\Reg S$.
The capped spatial curvature radius of a flat
product equals that of its residual factor. Integrating the product
radius estimate over a fixed Euclidean ball gives the residual
estimate. Compactness bounds the locally Lipschitz residual
potential, so weighted and ordinary energies are comparable.
The cutoff construction of \cite[Lemma~3.1]{Splitting} then gives \eqref{ii:lim:eq:capacity}.

The same spectral argument applies to a moving sequence whose limit
is a connected shrinking soliton; this soliton structure is established
in Proposition~\ref{ii:lim:prop:dini}. The complex and algebraic regularity
follows from Hallgren--Zhang \cite[Theorems~A and~B]{HZ}.

\begin{lemma}\label{ii:lim:lem:compact}
The residual factor $S$ is compact. Its time-$-1$ regular volume
$V_S$ satisfies $0<V_S\le V$.
\end{lemma}

\begin{proof}
Choose one negative time $a_0$ at which the $\Fconv$ convergence
is time-wise. Such times form a set of full measure after passing to a subsequence
\cite[Lemma~6.7]{BamlerCompactness}.
For $y,y'\in\Reg S_{a_0}$, fibre graphs give points
$y_i,y_i'\in F_q$ converging regularly to $(0,y),(0,y')$.
At a time-wise slice the regular convergence embeddings converge
uniformly on compact sets in the common slice metric
\cite[Theorem~9.31(d)]{BamlerCompactness}. Consequently
\[
 d_{S_{a_0}}(y,y')
 =\lim_i d_{g_i(a_0)}(y_i,y_i')\le D\sqrt{-a_0}.
\]
Density extends this bound to the residual completion. Exact
self-similarity gives $\diam S_{-1}\le D$.
The residual is a complete locally compact length space, hence
proper. Since it is bounded, it is compact.

At time $-1$, smooth fibre graphs over compact regular
domains give $V_S\le V$ by exhaustion. A nonempty regular domain
has positive volume, so $V_S>0$.
\end{proof}

\begin{proposition}\label{ii:lim:prop:no-loss}
Every residual factor of a fixed-fibre tangent satisfies $V_S=V$.
The same identity holds for a moving sequence whose limit is a
split shrinking soliton. Moreover, for the
horizontal determinant term of Subsection~\ref{ii:lim:subsec:geometry},
\begin{equation}\label{ii:lim:eq:Dvariance}
 \Var_{\nu_{q,s}}D_q(s)\longrightarrow0\qquad(s\to\infty).
\end{equation}
\end{proposition}

\begin{proof}
Use \eqref{ii:lim:eq:capacity}. For fixed $j$, transport $\chi_j$ to
$F_q$ through its fibre graph in $(X,g_i(-1))$, and extend by
zero outside a slightly larger regular graph domain. Denote the
result by $\chi_{j,i}$. Its compactly supported extension is smooth.
Smooth convergence gives convergence of its first and second
integrals and its Dirichlet energy. The fibre Poincar\'e inequality
therefore gives, as $i\to\infty$,
\[
 \frac1V\int_{\Reg S}\chi_j^2\dV_S
 -\left(\frac1V\int_{\Reg S}\chi_j\dV_S\right)^2
 \le\frac CV\int_{\Reg S}|\nabla\chi_j|^2\dV_S.
\]
Now let $j\to\infty$. If $b=V_S/V$, then
\[
 b-b^2\le0,\qquad 0<b\le1.
\]
Hence $b=1$. In particular, compact regular graph domains capture
all but an arbitrarily small amount of the ordinary fibre volume.

Put $t_i=T-\tau_i$, $s_i=-\log\tau_i$, and
\[
 \widetilde{\omega}_i=\tau_i^{-1}\omega(t_i),\qquad
 z_i=\tau_i^{-1/2}(w-w(q)).
\]
Using the Hermitian cometric and the block-inverse identity
\eqref{ii:lim:eq:schur}, we obtain
\[
\begin{aligned}
 G_i^{\alpha\bar\beta} :=\langle dz_i^\alpha,dz_i^\beta\rangle_{\widetilde{\omega}_i} =\tau_i\tau_i^{-1}
   \langle dw^\alpha,dw^\beta\rangle_{\omega(t_i)}
 =(H_{t_i}^{-1})^{\alpha\bar\beta}.
\end{aligned}
\]
Smooth convergence of the metrics and base maps gives
\[
 G_i\longrightarrow
 G_\infty
 :=\bigl(\langle dz_\infty^\alpha,dz_\infty^\beta
                         \rangle_{\widetilde{\omega}_\infty}\bigr).
\]
Since $\nabla dz_\infty^\alpha=0$, we have $dG_\infty=0$.
Thus $G_\infty$ is a constant positive-definite matrix on
$\C^m\times\Reg S$. On each compact regular fibre graph
$\iota_i:U\to F_q$,
\[
\begin{aligned}
 D_q(s_i)\circ\iota_i
 &=-\log\det(G_i\circ\iota_i)-\log\det\omega_Y(q)\\
 &\longrightarrow
 c:=-\log\det G_\infty-\log\det\omega_Y(q)
 \qquad\text{in }C^\infty(U).
\end{aligned}
\]
The constant $c$ is the same for all such graphs along this subsequence. By \eqref{ii:lim:eq:SQ}, $|D_q(s_i)|\le C$.
Volume preservation and exhaustion then imply, for each finite $p$,
\begin{equation}\label{ii:lim:eq:D-Lp}
 \int_{F_q}|D_q(s_i)-c|^p\,d\nu_{q,s_i}\longrightarrow0.
\end{equation}
Indeed, choose a smoothly bounded domain $U\Subset\Reg S$
with $\Vol(S\setminus U)<\varepsilon V$, and let $U_i$ be its fibre
graph. Smooth convergence and $V_S=V$ give
\[
 \nu_{q,s_i}(U_i)\longrightarrow\frac{\Vol(U)}V,
 \qquad
 \limsup_i\nu_{q,s_i}(F_q\setminus U_i)\le\varepsilon.
\]
After increasing the common bound so that $|c|\le C$, we obtain
\[
 \begin{aligned}
 \int_{F_q}|D_q(s_i)-c|^p\,d\nu_{q,s_i}
 &\le\sup_{U_i}|D_q(s_i)-c|^p
       +(2C)^p\nu_{q,s_i}(F_q\setminus U_i),\\
 \limsup_i\int_{F_q}|D_q(s_i)-c|^p\,d\nu_{q,s_i}
 &\le(2C)^p\varepsilon.
 \end{aligned}
\]
First let $i\to\infty$ and then $\varepsilon\searrow0$. Finally,
for any probability measure $\nu$ and real constant $c$,
\[
 \Var_\nu D
 =\int(D-c)^2\,d\nu-\left(\int(D-c)\,d\nu\right)^2
 \le\int(D-c)^2\,d\nu.
\]
Thus \eqref{ii:lim:eq:D-Lp} with $p=2$ controls the variance even though
the limiting constant $c$ may depend on the chosen subsequence.

Every sequence $s_i\to\infty$ has a splitting subsequence to which
this proof applies. If \eqref{ii:lim:eq:Dvariance} failed, a sequence with
variance bounded below would contradict \eqref{ii:lim:eq:D-Lp} with $p=2$.
\end{proof}

The compactness and volume arguments apply to moving sequences
once their limits are split shrinking solitons: the ambient
diameter bound and the fibre Poincar\'e inequality are uniform in
$q$, and the fibre graphs have the same regular convergence.

We give the argument over the regular product locus $Y^\circ$.
All smooth fibres under consideration are biholomorphic to the same
Fano manifold $F$. For uniform statements choose concentric coordinate
balls and a larger product chart with
\[
 K_0=\overline{B_0}\Subset B_1\Subset Y^\circ,
 \qquad \Phi^{-1}(B_1)\cong B_1\times F.
\]
The closed ball $K_0$ is compact and connected. Constants below may
depend on the larger chart and its buffer. The local estimates and
limiting heat measures established above apply
uniformly to centres in $K_0$. Every compact subset of $Y^\circ$ is
covered by finitely many such closed balls. Thus the conclusions
proved uniformly on $K_0$ imply the stated compact-local conclusions;
for a Fano bundle over a compact smooth base they are global.

\subsection{Anticanonical sections and polarized limits}\label{ii:lim:subsec:sections}

A compact residual factor carries an anticanonical polarization
by \cite{HZ}. To relate it to the original fibre, however, one
must show that the complete spaces of sections converge, including
across the singular set. Extension to ambient tubes gives the
uniform section bounds through the estimates of \cite{JST}, and
$\bar\partial$ estimates provide approximation in the opposite
direction. Preservation of volume then keeps orthonormal sections
independent in the limit and determines the central fibre of the
resulting degeneration.

\begin{lemma}\label{ii:lim:lem:qfano}
Every compact residual $S$ of a split shrinking limit of the rescaled flow is a
projective klt $\Q$-Fano variety. Its soliton current belongs to $2\pi c_1(S)$,
has bounded local potentials, and has full anticanonical mass.
In particular,
\begin{equation}\label{ii:lim:eq:degree}
 (-K_S)^{n-m}=(-K_F)^{n-m}.
\end{equation}
\end{lemma}

\begin{proof}
Hallgren--Zhang \cite[Theorem~A and Definition~7.1]{HZ} give a
polarized Fano fibration $p:Z\to A$, with $A$ affine and
$-K_Z$ relatively ample and $\Q$-Cartier. Holomorphic functions
on compact connected normal $S$ are constant, so the slice
$\{0\}\times S$ lies in a projective fibre of $p$. It is algebraic
by Chow's theorem \cite{Chow}. The product identification gives
$K_Z|_{\{0\}\times S}=K_S$; hence an ample Cartier multiple
of $-K_Z$ restricts to an ample multiple of $-K_S$.
Since $S$ is klt by \cite[Theorem~3.3]{Splitting}, it is $\Q$-Fano.

The bounded continuous anticanonical weights of
\cite[Section~2.3 and Proposition~4.6]{HZ} restrict to the slice
and give $[\omega_S]=2\pi c_1(S)$. On a resolution
$\rho:\widetilde S\to S$, write
\[
 \rho^*\omega_S=\theta+\ddc\varphi,\qquad
 [\theta]=2\pi c_1(\rho^*(-K_S)),\qquad \varphi\in L^\infty.
\]
The Bedford--Taylor product has no mass on the exceptional set;
its telescoping identity and Stokes' theorem \cite{BedfordTaylor} give
\[
 \int_{\Reg S}\omega_S^{n-m}
 =\int_{\widetilde S}(\theta+\ddc\varphi)^{n-m}
 =\int_{\widetilde S}\theta^{n-m}=(2\pi)^{n-m}(-K_S)^{n-m}.
\]
Thus $V_S=(4\pi)^{n-m}(-K_S)^{n-m}/(n-m)!$, and $V_S=V$ proves
\eqref{ii:lim:eq:degree}. Finally, the regular soliton equation
is the Monge--Amp\`ere equation for these weights. Its right-hand
side is the anticanonical volume measure times a bounded
exponential factor, and is locally integrable by the klt
condition. Both sides have zero mass on the singular set, so
the equation extends to all of $S$ as the weak soliton equation.
\end{proof}

Fix a sequence $(q_i,t_i)$ with $\tau_i=T-t_i\to0$ that admits a
split shrinking limit. In a product chart, use the normalized tube
$\Omega_i=\Phi^{-1}(|z|<R_0)$, where
$z=(w-w(q_i))/\sqrt{\tau_i}$ and $R_0$ is a fixed small number.
Let $\widetilde{\omega}_i=\tau_i^{-1}\omega(t_i)$ and
$\eta_i=\widetilde{\omega}_i|_{F_{q_i}}$. On a slightly larger tube choose $v_i$
as in \eqref{ii:lim:eq:tube-v}. Let $G_i$ denote the determinant metric on
$-K_{\Omega_i}$ induced by $\widetilde{\omega}_i$, and put
\begin{equation}\label{ii:lim:eq:Hambient}
 H_i=e^{-v_i}G_i,\qquad \sqrt{-1}\Theta(H_i)=\widetilde{\omega}_i.
\end{equation}
The base volume form $dz^1\wedge\cdots\wedge dz^m$ fixes the
adjunction identification $(-K_{\Omega_i})|_{F_{q_i}}=-K_{F_{q_i}}$.
Write $L_i=-K_{F_{q_i}}$ and $H_{F,i}=H_i|_{F_{q_i}}$ in this
identification. If $G_{F,i}$ is the fibre determinant metric and
$\mathsf S_i$ is the quotient matrix relative to $z$, then
\begin{equation}\label{ii:lim:eq:a-metric}
 H_{F,i}=a_iG_{F,i},\qquad
 a_i=e^{-v_i}\det \mathsf S_i,\qquad c\le a_i\le C,\qquad
 \sqrt{-1}\Theta(H_{F,i})=\eta_i.
\end{equation}
The two-sided bound uses the scaled coordinates: $\mathsf S_i$ equals the
unnormalized quotient in the original coordinates.

\begin{proposition}\label{ii:lim:prop:section-sup}
For every fixed integer $r\ge1$ there is a constant $C_r$, independent
of $i$, such that every $\sigma\in H^0(F_{q_i},L_i^r)$ satisfies
\begin{equation}\label{ii:lim:eq:section-sup}
 \sup_{F_{q_i}}|\sigma|_{H_{F,i}^r}^2
 \le C_r\int_{F_{q_i}}|\sigma|_{H_{F,i}^r}^2\dV_{\eta_i}.
\end{equation}
\end{proposition}

\begin{proof}
The tube is weakly pseudoconvex: pull back a proper smooth
plurisubharmonic exhaustion of its base ball; properness of $\Phi$
makes the pullback proper. It is K\"ahler with metric $\widetilde{\omega}_i$.
Set $E_i=(-K_{\Omega_i})^{r+1}$, with metric $H_i^{r+1}$.
Its curvature is $(r+1)\widetilde{\omega}_i$. By adjunction, $\sigma$ is the
restriction of the bundle $K_{\Omega_i}\otimes E_i=-rK_{\Omega_i}$.

Apply the adjoint extension theorem of Zhou--Zhu
\cite[Theorem~1.1]{ZhouZhu} with
\[
 \psi=m\log|z|^2,\qquad \alpha_0=\alpha_1=0,\qquad
 R(t)=e^{-t}.
\]
Shrink $R_0$ once so that $\psi\le-1$, and take their continuous
function $\alpha=-1/2$. The weight has neat analytic singularities and multiplier
ideal equal to the ideal of the central fibre: a germ vanishing
there is $O(|z|)$, and its normal squared-norm integral is bounded
by $C\int_0^\epsilon\rho\,d\rho$, whereas a nonzero restriction
gives $\int_0^\epsilon\rho^{-1}d\rho=\infty$. The normal integral
for $e^{-(1-\delta)\psi}$ is
$\int_0^\epsilon\rho^{2m\delta-1}d\rho<\infty$ for $0<\delta<1$,
so the singularities are log canonical along the fibre.
The function $R$ is admissible by \cite[Remark~1.3]{ZhouZhu},
$C_R=1$, and $e^\psi R(\psi)=1$. Both curvature inequalities in
that theorem hold because $\ddc\psi\ge0$ and the bundle curvature
is positive.

The theorem supplies a holomorphic extension $\Sigma$ to $\Omega_i$.
To check its norms, the canonical metric induced by $\widetilde{\omega}_i$ gives
\begin{equation}\label{ii:lim:eq:adjointnorm}
 |\Sigma|^2_{K_{\Omega_i}\otimes H_i^{r+1}}
       =e^{-v_i}|\Sigma|^2_{H_i^r}.
\end{equation}
The residue measure for $\psi$ is
\begin{equation}\label{ii:lim:eq:residuemeasure}
 dV_{\widetilde{\omega}_i}[\psi]=c_m\det \mathsf S_i\,dV_{\eta_i},
\end{equation}
where $c_m>0$ depends only on dimension and the fixed volume
convention. To check both identities in local frames, let $e$ be
the coordinate frame of $-K_{\Omega_i}$ and put
$J_i=|e|_{G_i}^2=\det(\widetilde{\omega}_{i,I\bar J})$.
For $\Sigma=\sigma e^{\otimes r}$ the adjoint identification is
$\sigma e^*\otimes e^{\otimes(r+1)}$, and hence
\[
 \begin{aligned}
 |\Sigma|^2_{K_{\Omega_i}\otimes H_i^{r+1}}
 &=|\sigma|^2J_i^{-1}(e^{-v_i}J_i)^{r+1}\\
 &=e^{-v_i}|\sigma|^2(e^{-v_i}J_i)^r
  =e^{-v_i}|\Sigma|_{H_i^r}^2.
 \end{aligned}
\]
For the residue measure, the Schur identity gives, at $z=0$,
$dV_{\widetilde{\omega}_i}=\det\mathsf S_i\,dV_{\eta_i}\,dV_z$.
Write $dV_z=c'_m\rho^{2m-1}d\rho\,d\vartheta$, with the angular
measure normalized to have total mass one. The normal integral is
\[
 \int_{\{t<m\log|z|^2<t+1\}}|z|^{-2m}\,dV_z
 =c'_m\int_{e^{t/(2m)}}^{e^{(t+1)/(2m)}}\frac{d\rho}{\rho}
 =\frac{c'_m}{2m}.
\]
As $t\to-\infty$, the annuli shrink to the central fibre. Against
a continuous test function, the remaining smooth volume density
therefore tends to its value at $z=0$. This proves
\eqref{ii:lim:eq:residuemeasure}, with $c_m=c'_m/(2m)$, and shows that
no power of $\tau_i$ remains in the extension norm.
Equations~\eqref{ii:lim:eq:adjointnorm}--\eqref{ii:lim:eq:residuemeasure},
bounded $v_i$, and bounded $\det \mathsf S_i$ give
\begin{equation}\label{ii:lim:eq:extension-L2}
 \int_{\Omega_i}|\Sigma|^2_{H_i^r}\dV_{\widetilde{\omega}_i}
 \le C\int_{F_{q_i}}|\sigma|^2_{H_{F,i}^r}\dV_{\eta_i}.
\end{equation}

The curvature of $H_i^r$ is $r\widetilde{\omega}_i$. The Bochner inequality
used in \cite[Proposition~8.4, proof of equation~(8.4)]{JST}
therefore applies to the regularized norm
$w_\eps=(|\Sigma|_{H_i^r}^2+\eps^2)^{1/2}$ and gives
\[
 \Delta_{\widetilde{\omega}_i}w_\eps\ge-\frac{rn}{2}w_\eps.
\]
The local Moser argument in the same proof uses this inequality,
a Sobolev bound, and cutoffs with controlled gradients.
Here the Sobolev bound is \eqref{ii:lim:eq:ambient-sobolev}.
Nested base cutoffs of widths comparable to $2^{-j}$ pull back to
cutoffs with $|\partial\zeta_j|_{\widetilde{\omega}_i}\le C2^j$ by the Schwarz bound \eqref{ii:eq:basic}.
Thus the cited argument, followed by $\eps\searrow0$, gives
\[
 \sup_{\Omega_i'}|\Sigma|^2_{H_i^r}\le C_r
       \int_{\Omega_i}|\Sigma|^2_{H_i^r}\dV_{\widetilde{\omega}_i},
\]
where the fixed inner tube still contains the entire central fibre.
Combining with \eqref{ii:lim:eq:extension-L2} proves the proposition.
\end{proof}

\begin{remark}\label{ii:lim:rem:jst-section-estimate}
The local $L^2$-to-$L^\infty$ estimate above is the Moser estimate
of Jian--Song--Tian \cite[Proposition~8.4]{JST}, applied with the
uniform ambient Sobolev inequality and the pulled-back base
cutoffs. The fibrewise extension estimate
\eqref{ii:lim:eq:extension-L2} verifies its input here. The
complete convergence of anticanonical section spaces, including
the behavior near the singular set, is proved below.
\end{remark}

\begin{lemma}\label{ii:lim:lem:hormander}
If $e$ is a smooth $\bar\partial$-closed $(0,1)$-form with values
in $L_i^r$, there is $u$ with $\bar\partial u=e$ and
\begin{equation}\label{ii:lim:eq:hormander}
 \int_{F_{q_i}}|u|^2_{H_{F,i}^r}a_i\dV_{\eta_i}
 \le\frac1{r+1}\int_{F_{q_i}}|e|^2_{H_{F,i}^r,\eta_i}a_i\dV_{\eta_i}.
\end{equation}
Thus the corresponding unweighted $L^2$ estimate is uniform in $i$.
\end{lemma}

\begin{proof}
Use $L_i^r=K_{F_{q_i}}\otimes L_i^{r+1}$ and apply
\cite[Theorem~8.4 and formula~(7.7)]{DemaillyL2} to $(n-m,1)$-forms
with values in $(L_i^{r+1},H_{F,i}^{r+1})$ on the compact
K\"ahler fibre. The coefficient curvature is $(r+1)\eta_i$,
whose commutator with $\Lambda$ on these forms is
$(r+1)((n-m)+1-(n-m))\mathrm{Id}=(r+1)\mathrm{Id}$.
The adjoint metric is, by \eqref{ii:lim:eq:a-metric},
\[
 G_{F,i}^{-1}H_{F,i}^{r+1}
 =G_{F,i}^{-1}(a_iG_{F,i})^{r+1}=a_iH_{F,i}^r.
\]
Thus the cited estimate has measure $a_i\dV_{\eta_i}$ in
the $H_{F,i}^r$ norm and inverse curvature factor $(r+1)^{-1}$,
which is \eqref{ii:lim:eq:hormander}. The bounds $c\le a_i\le C$
give its unweighted form without an intrinsic Ricci bound.
\end{proof}

\label{ii:lim:sec:algebraic}

We identify the limiting metrics and record the integral
convergence used for the section spaces. Normalize the residual
potential by
\[
 \Ric(\omega_S)+\ddc u_S=\omega_S,\qquad
 d\nu_S=V^{-1}dV_S,\qquad \int_Se^{-u_S}d\nu_S=1.
\]
It is bounded because it is locally Lipschitz on compact $S$.
On regular fibre graphs, the trace equations for $v_i$ and
$f_i=f_{q_i}(s_i)$, their uniform bounds, and interior elliptic
estimates give smooth local limits $v_\infty|_S$ and $f_\infty$.
Their differences from $u_S$ are bounded and pluriharmonic.
For either difference $b$, the capacity cutoffs satisfy
\[
 \int\chi_j^2|\nabla b|^2dV_S
 =-2\int\chi_j b\langle\nabla\chi_j,\nabla b\rangle dV_S
 \le\tfrac12\int\chi_j^2|\nabla b|^2dV_S
       +2\|b\|_\infty^2\int|\nabla\chi_j|^2dV_S.
\]
Letting $j\to\infty$ shows that $b$ is constant on connected
$\Reg S$.

Here is the common integral convergence argument. Put
$dV_i=\eta_i^{n-m}/(n-m)!$ and $dV_\infty=\omega_S^{n-m}/(n-m)!$, both of mass
$V_F$. If $\zeta_i\to \zeta_\infty$ smoothly on regular graphs and
$|\zeta_i|,|\zeta_\infty|\le M$, then for a compact regular domain $U$
and its graph $U_i$,
\[
 \begin{aligned}
 \left|\int_{F_{q_i}}\zeta_i\,dV_i-\int_S\zeta_\infty\,dV_\infty\right|
 &\le\left|\int_{U_i}\zeta_i\,dV_i-\int_U\zeta_\infty\,dV_\infty\right|\\
 &\quad+M\bigl(\Vol_{dV_i}(F_{q_i}\setminus U_i)
                     +\Vol_{dV_\infty}(S\setminus U)\bigr).
 \end{aligned}
\]
The first term tends to zero; no loss of volume makes the last
line arbitrarily small as $U$ exhausts $\Reg S$. Apply this to
$e^{-f_i}$, bounded by the uniform bound for $f_i$. Dividing by
$V_F$ gives $\int_Se^{-f_\infty}d\nu_S=1$. Since
$f_\infty=u_S+c$, this equality reads $1=e^{-c}$, and hence
\begin{equation}\label{ii:lim:eq:residual-potential}
 f_\infty=u_S,\qquad |u_S|\le C.
\end{equation}
The quotient determinant converges to a positive constant on the
product slice. Therefore \eqref{ii:lim:eq:a-metric} gives
\begin{equation}\label{ii:lim:eq:soliton-Hlimit}
 H_{F,i}\longrightarrow H_S^{\mathrm{tube}}=c_0e^{-u_S}G_S
       \quad\text{locally smoothly on }\Reg S,
\end{equation}
where $c_0>0$ includes that determinant and the constant in
$v_\infty|_S-u_S$. The bounded nonvanishing anticanonical frames
of \cite[proof of Proposition~4.8, equation~(4.7)]{HZ} and the
continuous weights of \cite[Proposition~4.6]{HZ} restrict to
$S$. Thus every Cartier power of $H_S^{\mathrm{tube}}$ has
bounded continuous weights, and its global sections have bounded
norm. Recall that $dV_\infty=2^{-(n-m)}dV_S$.

\begin{proposition}\label{ii:lim:prop:h0}
For every sufficiently divisible integer $r>0$,
\begin{equation}\label{ii:lim:eq:h0}
 h^0(F,-rK_F)=h^0(S,-rK_S).
\end{equation}
There are bases on the left which converge locally smoothly on
fibre graphs to a basis on the right.
\end{proposition}

\begin{proof}
Take $r$ divisible by the Cartier index of $S$, and first let
$\sigma\in H^0(S,-rK_S)$. Transfer $\chi_j\sigma$ to the
corresponding regular fibre graph and extend by zero, obtaining
a smooth global section $a_{j,i}$ on $F_{q_i}$. The bundle
identification $T_i$ is obtained by projecting between the
converging $(1,0)$ tangent bundles and taking determinants.
On the fixed support of $\chi_j$, write
$\bar\partial_i^{\rm tr}=T_i^{-1}\bar\partial T_i$. Then
\[
 T_i^{-1}\bar\partial a_{j,i}
 =(\bar\partial\chi_j)\sigma
   +(\bar\partial_i^{\rm tr}-\bar\partial)(\chi_j\sigma).
\]
The coefficients of the last first-order operator tend smoothly
to zero. Consequently
\begin{equation}\label{ii:lim:eq:barerror}
 \limsup_{i\to\infty}\|\bar\partial a_{j,i}\|_{L^2}^2
 \le C_r\|\sigma\|_{L^\infty((H_S^{\mathrm{tube}})^r)}^2
             \int_{\Reg S}|\nabla\chi_j|^2dV_S.
\end{equation}
Either the weighted or the uniformly comparable unweighted norms
may be used. The datum is $\bar\partial$-closed, so
Lemma~\ref{ii:lim:lem:hormander} supplies corrections $u_{j,i}$
with $\bar\partial u_{j,i}=\bar\partial a_{j,i}$ and this same
$L^2$ bound up to a fixed factor.

Choose the cutoffs with energy at most $2^{-4j}$ and increasing
integers $I_j$ such that, for $i\ge I_j$, the fixed-support
operator errors have $L^2$ norm at most $2^{-j}$ and are at most
$2^{-j}$ in $C^j$ on the first $j$ members of a regular
exhaustion, lying where $\chi_j=1$. With
$j(i)=\max\{j:I_j\le i\}$, the corrections satisfy
$\|u_{j(i),i}\|_{L^2}\le C_r2^{-j(i)}\to0$.
On regular domains $U\Subset U'$, interior elliptic estimates
and Sobolev embedding give, for every $\ell$,
\[
 \|u_{j(i),i}\|_{C^\ell(U)}
 \le C_{\ell,U,U'}\bigl(\|u_{j(i),i}\|_{L^2(U')}
       +\|\bar\partial a_{j(i),i}\|_{H^{\ell+n-m+1}(U')}\bigr)
 \longrightarrow0.
\]
The constants are uniform because the local coefficients
converge smoothly. Hence the holomorphic sections
$\sigma_i=a_{j(i),i}-u_{j(i),i}$ converge locally smoothly to
$\sigma$. Applying the same construction to a finite independent
family preserves independence for large $i$, and proves
\begin{equation}\label{ii:lim:eq:h0-first}
 h^0(F,-rK_F)\ge h^0(S,-rK_S).
\end{equation}

For the reverse inequality, take an orthonormal basis
$\sigma_i^0,\ldots,\sigma_i^N$ of $H^0(F_{q_i},-rK_{F_{q_i}})$
for $H_{F,i}^r$ and $dV_i$. Its dimension is fixed because
$F_{q_i}\cong F$. Proposition~\ref{ii:lim:prop:section-sup}
and regular elliptic estimates give locally smooth limits
$\sigma^a$ on $\Reg S$. Since $-rK_S$ is Cartier, normal
Hartogs extension in its local holomorphic frames makes these
global sections. Their pointwise pairings are uniformly bounded
by the section estimate. The integral convergence proved above,
applied to $\langle\sigma_i^a,\sigma_i^b\rangle$, therefore gives
\[
 \int_S\langle\sigma^a,\sigma^b\rangle_{(H_S^{\mathrm{tube}})^r}
                  \frac{\omega_S^{n-m}}{(n-m)!}=\delta_{ab}.
\]
Thus the limits are independent. This proves the reverse
inequality and, together with \eqref{ii:lim:eq:h0-first}, their
completeness. Constants may depend on $r$.
\end{proof}

\begin{proposition}\label{ii:lim:prop:flat}
There is a projective flat $\Q$-Gorenstein family
$\mathcal X\to(C,0)$ over a smooth pointed curve, isomorphic to
$F\times(C\setminus\{0\})$ away from $0$, with central fibre $S$
and relatively ample $-K_{\mathcal X/C}$.
\end{proposition}

\begin{proof}
Choose $r$ sufficiently divisible that $-rK_F$ and $-rK_S$
are very ample. Proposition~\ref{ii:lim:prop:h0} gives complete
bases defining embeddings $\Phi_i:F\hookrightarrow\PP^N$
converging on regular graphs to an embedding of $S$. Take a
Hilbert-scheme limit $W$ of $\Phi_i(F)$. The closed universal
family implies $S\subset\supp W$; since $S$ is reduced, every
function in $\mathcal I_W$ vanishes on $S$, so
$\mathcal I_W\subset\mathcal I_S$. For all sufficiently large $a$,
\[
 P_W(a)=P_F(a)=h^0(F,-arK_F)=h^0(S,-arK_S)=P_S(a),
\]
again by Proposition~\ref{ii:lim:prop:h0}. The exact sequence
\[
 0\longrightarrow\mathcal I_S/\mathcal I_W
 \longrightarrow\mathcal O_W\longrightarrow\mathcal O_S
 \longrightarrow0
\]
shows that $\mathcal I_S/\mathcal I_W$ has zero Hilbert
polynomial, hence is zero. Thus $W=S$ scheme-theoretically.

The points $[\Phi_i(F)]$ lie in one $\operatorname{PGL}(N+1)$
orbit. Choose an algebraic curve through $[S]$ in its closure
with generic point in that orbit. The generic orbit-map fibre
has a closed point over a finite field extension. After the
corresponding finite base change and normalization, the generic
point therefore lifts to $\operatorname{PGL}(N+1)$. Represent
the lift by an invertible matrix over the curve's function field,
and shrink around $0$ so that the matrix and its inverse are
regular away from $0$. Pulling back the universal family gives
the required projective flat family, with a polarized
trivialization over the punctured curve.

We check the $\Q$-Gorenstein property. All fibres are normal:
they are $S$ at $0$ and $F$ elsewhere. Flatness over the smooth
curve gives normality of the total space \cite[Tag~0C22]{Stacks};
it is integral because the generic fibre is integral and flatness
excludes vertical components. Let $H$ be the relative hyperplane
divisor and $t$ a uniformizer at $0$. On the punctured family
$H\sim-rK_{\mathcal X/C}$. Using the Weil-divisor correspondence
on the normal total space \cite[Tag~0EBM]{Stacks}, choose a
rational function $\xi$ representing this equivalence. Then
\[
 E=H+rK_{\mathcal X/C}-\operatorname{div}(\xi)
\]
is supported on the integral central fibre, so $E=bS$ for an
integer $b$. Reducedness of that fibre implies that the valuation
of $t$ at its generic point is one; hence $\operatorname{div}(t)=S$.
Consequently
\[
 H+rK_{\mathcal X/C}=\operatorname{div}(\xi t^b),\qquad
 \mathcal O_{\mathcal X}(-rK_{\mathcal X/C})
       \cong\mathcal O_{\mathcal X}(H).
\]
Thus $-rK_{\mathcal X/C}$ is Cartier and relatively ample.
Adjunction on the smooth locus of each normal fibre, followed
by reflexive extension across codimension two, identifies its
restriction with $-rK_F$ or $-rK_S$. All fibres are klt, so this
is a $\Q$-Gorenstein family of $\Q$-Fano varieties.
\end{proof}

\begin{proposition}\label{ii:lim:prop:soliton-static}
Every fixed-fibre tangent has a compact residual soliton $S$ with
\[
 \Vol(S,g_S)=V,\qquad (-K_S)^{n-m}=(-K_F)^{n-m},\qquad |u_S|\le C.
\]
There is a polarized $\Q$-Gorenstein degeneration of $F$ to $S$,
and complete anticanonical bases converge in each sufficiently
divisible degree. The same conclusions hold for moving sequences
whose limits are split shrinking solitons. On each fixed fibre,
$B_q(s):=\Var_{\nu_{q,s}}D_q(s)\to0$.
\end{proposition}
\begin{proof}
Compactness and volume preservation follow from
Lemma~\ref{ii:lim:lem:compact} and Proposition~\ref{ii:lim:prop:no-loss}.
Lemma~\ref{ii:lim:lem:qfano} and \eqref{ii:lim:eq:residual-potential} give the
degree and potential bounds. Propositions~\ref{ii:lim:prop:h0}
and~\ref{ii:lim:prop:flat} give the section convergence and degeneration.
The last assertion is the variance conclusion of
Proposition~\ref{ii:lim:prop:no-loss}.
\end{proof}

\subsection{Entropy and moving fibres}\label{ii:lim:sec:uniform}\label{ii:lim:sec:soliton-moving}\label{ii:lim:sec:moving}

A tangent at a fixed terminal point is a shrinker, but this does
not by itself settle the case of base points moving with the
rescaling parameter. The necessary uniformity comes from entropy.
We first express the entropy of a split shrinker as its volume
contribution minus a relative entropy. When the fixed-fibre
terminal entropies agree, continuity and monotonicity give uniform
convergence by compactness; entropy rigidity then makes every
moving limit a shrinker as well.

Suppose $Z=\C^m\times S$ is any split shrinking soliton with compact
residual $S$ and real regular volume $V_S$. At time $-1$ its
distinguished probability is the product of a Euclidean Gaussian
and a residual probability $\sigma_S$; denote the full product probability by
$\sigma_Z$. Normalize the residual
potential $u$ by
\[
 d\sigma_S=V_S^{-1}e^{-u}\dV_S,\qquad
 \int_{\Reg S}e^{-u}\dV_S=V_S.
\]
After translating the Euclidean origin, the full density potential is
\[
 f_Z(z,y)=\frac{|z|^2}{4}+u(y)+\log\frac{V_S}{(4\pi)^{n-m}}.
\]
Indeed, the normalized product probability is
\[
 d\sigma_Z=d\gamma_m\,d\sigma_S,\qquad
 d\gamma_m=(4\pi)^{-m}e^{-|z|^2/4}\,dz,
 \qquad \int_{\C^m}\frac{|z|^2}{4}\,d\gamma_m=m.
\]
The last equality is the sum of the second moments of the $2m$
real Gaussian coordinates, each of variance $2$. Hence
\[
 \begin{aligned}
 \int f_Z\,d\sigma_Z-n
 &=m+\int_Su\,d\sigma_S
       +\log\frac{V_S}{(4\pi)^{n-m}}-n\\
 &=\log\frac{V_S}{(4\pi)^{n-m}}-(n-m)+\int_Su\,d\sigma_S.
 \end{aligned}
\]
For probability measures $d\sigma=\rho\,dP$, the relative entropy
(Kullback--Leibler divergence) is defined by
\[
 \Ent(\sigma\mid P):=\int\rho\log\rho\,dP.
\]
Here $dP=V_S^{-1}dV_S$ and $\rho=e^{-u}$, so
$d\sigma_S=\rho\,dP$, $\int\rho\,dP=1$, and
\[
 \Ent(\sigma_S\mid P)=-\int u\,d\sigma_S\ge0.
\]
Here the inequality follows by applying convexity to $x\log x$;
its strict convexity makes equality equivalent to $\rho=1$
almost everywhere. The Nash entropy $\int f_Z\,d\sigma_Z-n$
therefore satisfies the exact identity
\begin{equation}\label{ii:lim:eq:entropy-volume}
 \Nash_Z=\log\frac{V_S}{(4\pi)^{n-m}}-(n-m)
  -\Ent\!\left(\sigma_S\,\middle|\,V_S^{-1}dV_S\right)
 \le\log\frac{V_S}{(4\pi)^{n-m}}-(n-m).
\end{equation}
Equality holds exactly when $u=0$, in which case the metric is
K\"ahler--Einstein. The locally Lipschitz potential on compact $S$
makes all these integrals finite. 

\begin{proposition}\label{ii:lim:prop:dini}
Assume that the limiting Nash entropy at every fixed-fibre
limiting point over $B_1$ is the same constant $\Theta$. For $0<r<T$,
where $r$ denotes backward time and $q\in K_0$, define
\[
 d\mu_{T-r}^q=(4\pi r)^{-n}e^{-f_{q,r}}\dV_{g(T-r)},\qquad
 \Nash_q(r)=\int_Xf_{q,r}\,d\mu_{T-r}^q-n.
\]
Then
\begin{equation}\label{ii:lim:eq:dini}
 \lim_{r\searrow0}\sup_{q\in K_0}|\Nash_q(r)-\Theta|=0.
\end{equation}
For arbitrary $q_i\in K_0$ and $\tau_i\searrow0$, every
$\Fconv$-subsequential limit of the rescaled flows with measures
$\mu_{T+\tau_i a}^{q_i}$ is a shrinking soliton with distinguished
Nash entropy $\Theta$.
\end{proposition}

\begin{proof}
At each fixed $r>0$, Proposition~\ref{ii:lim:prop:heat} gives
continuous $C^\infty$ dependence of the positive probability density
on $q$, hence continuity of $\Nash_q(r)$.
Given $q_i\in K_0$ and $\tau_i\searrow0$, pass to a subsequence
with $q_i\to q_\infty\in K_0$. For fixed $b>0$ and $0<r<T$,
we have $b\tau_i<r$ for large $i$, and Nash monotonicity
\cite[Theorem~2.37]{BamlerStructure} gives
\[
 \Nash_{q_i}(r)\le\Nash_{q_i}(b\tau_i)\le\Theta.
\]
By continuity at fixed scale,
\[
 \Nash_{q_\infty}(r)
 \le\liminf_i\Nash_{q_i}(b\tau_i)
 \le\limsup_i\Nash_{q_i}(b\tau_i)\le\Theta.
\]
Letting $r\searrow0$ proves $\Nash_{q_i}(b\tau_i)\to\Theta$.
Since every sequence has such a subsequence, compactness also
gives the uniform convergence in \eqref{ii:lim:eq:dini}.

By parabolic invariance, the rescaled entropy at backward time $b$
is $\Nash_{q_i}(b\tau_i)$. The original compact flow supplies
the uniform entropy lower bound
\cite[equation~(4.34) and Proposition~4.35]{BamlerStructure}.
Entropy convergence and soliton rigidity
\cite[Theorems~2.14 and~2.18]{BamlerStructure} therefore imply
that every subsequential limit is a shrinking soliton with
Nash entropy $\Theta$.
\end{proof}

We apply \cite[Section~3]{Splitting} to the moving soliton of
Proposition~\ref{ii:lim:prop:dini}. Use
$g_i(a)=\tau_i^{-1}g(T+\tau_i a)$ and
$\mu_{i,a}=\mu_{T+\tau_i a}^{q_i}$. After passing to a subsequence,
$q_i\to q_\infty$ in a common regular product chart. Set
\begin{equation}\label{ii:lim:eq:moving-maps}
 \mathcal B_i=\tau_i^{-1/2}(w-w(q_i))\circ\Phi.
\end{equation}
Equations~\eqref{ii:lim:eq:SQ} and \eqref{ii:lim:eq:basemoment}
give uniform Lipschitz, second-moment and full-rank Gram bounds;
the original flow supplies the entropy lower bounds at all centres
used in regular convergence.

To check centering, choose a base cutoff $\zeta$ supported in
$B_1$ and equal to one near $K_0$. Extend
$\psi_i=\Phi^*(\zeta(w-w(q_i)))$ by zero and put
$\widetilde{\mathcal B}_i=\tau_i^{-1/2}\psi_i$.
The unscaled base $C^2$ norms are uniform, so Schwarz gives
$|\Delta_{g(t)}\psi_i|\le C$. Heat duality, first for point
kernels and then for their limits, gives, for every fixed $b>0$, componentwise
\[
 \int_X\psi_i\,d\mu_{T-b\tau_i}^{q_i}
 =\int_{T-b\tau_i}^{T}\!\int_X
        \Delta_{g(t)}\psi_i\,d\mu_t^{q_i}\,dt,
 \qquad
 \left|\int_X\widetilde{\mathcal B}_i
                   \,d\mu_{T-b\tau_i}^{q_i}\right|
 \le Cb\sqrt{\tau_i}.
\]
Here the limiting pole value is zero. Since $|\psi_i|^2\le C\Phi^*\rho_{q_i}$,
\eqref{ii:lim:eq:basemoment} also gives
\[
 \int_X|\widetilde{\mathcal B}_i|^2\,d\mu_{T-b\tau_i}^{q_i}\le Cb,
 \qquad
 \int_{\{|\widetilde{\mathcal B}_i|>R\}}
       |\widetilde{\mathcal B}_i|\,d\mu_{T-b\tau_i}^{q_i}
 \le Cb/R.
\]
Bounded truncation followed by $R\to\infty$ passes the first
moments to the limit. The limiting map is therefore centered;
placing the Euclidean origin at the Gaussian center removes its
affine constant.

The sharp kernel Poincar\'e inequality
\cite[Lemmas~3.3--3.4]{Splitting}, applied to each limiting measure
and rescaled to time $-1$, gives, for every smooth real test function $u$,
\[
 \Var_{\mu_{i,-1}}u\le
       2\int_X|\nabla u|_{g_i(-1)}^2\,d\mu_{i,-1}.
\]
It passes to compact regular test functions on the moving soliton.
The point-kernel diagonal below and the entropy bound verify the
inputs of \cite[Proposition~3.2, equations~(3.3)--(3.12)]{Splitting};
exact self-similarity gives the spatial inverse-radius estimate
and hence the cutoffs of \cite[Lemma~3.1]{Splitting}.
These are precisely the inputs, together with the soliton
identities, for the $1/2$-eigenspace argument in
\cite[Theorem~3.1 and Propositions~3.4--3.5]{Splitting}:
compactness of the holomorphic maps and the sharp inequality above
give a full-rank parallel limiting map. The intrinsic product
theorem \cite[Theorem~3.3]{Splitting} extends the splitting across
the normal singular set. Thus $Z=\C^m\times S$, and the limit of
\eqref{ii:lim:eq:moving-maps} is an invertible constant linear map
composed with the projection.

For the point-kernel formulations of \cite[Section~2]{Splitting}
and \cite[Section~2.2, setting~(I)]{HZ}, choose approximating poles
at $T_i<T$ with $(T-T_i)/\tau_i\to0$. By a diagonal choice their
scaled $W_1$ error at the latest time of each compact negative
interval tends to zero; backward contraction controls the whole
interval. The time translation to the limiting-measure rescalings
is $o(1)$, so the limit and its regular convergence agree.

\section{Proof of Theorem 1.2}\label{ii:sec:curvature}

The Einstein assumption now provides the rigidity needed to
identify the residual factor. The Ding functional has a lower
bound, while the error caused by the horizontal geometry tends
to zero in fibrewise variance. It follows that each fixed fibre
has an Einstein tangent. The entropy of this tangent determines
all moving limits, and K-polystable separatedness identifies
their residuals with the original fibre. Smooth convergence then
yields the full curvature and intrinsic diameter estimates.

\subsection{Entropy and identification of the limiting fibre}\label{ii:subsec:einstein}

The induced flow on a fibre differs from its intrinsic
K\"ahler--Ricci flow by a horizontal determinant term. Its vanishing
variance is enough to recover the part of Ding dissipation needed
here: there are times at which the intrinsic Ricci potential tends
to zero. The corresponding tangent is Einstein and attains the
entropy bound determined by the fibre volume. Equality in this
bound forces all moving residuals to be Einstein, and the
anticanonical degeneration identifies their complex structure.

Assume that $F$ admits a K\"ahler--Einstein metric $\omega_F$,
normalized by $\Ric(\omega_F)=\omega_F$, and put $g_F=2g_{\omega_F}$.
Then $\Ric(g_F)=\tfrac12g_F$. We use the normalized fibre metrics $h_q(s)$, Ricci potentials
$f_q$, determinant terms $D_q$, and probabilities $\nu_{q,s}$
of Section~\ref{ii:lim:sec:geometry}. The estimates give
\[
 |f_q|\le K,\qquad
 B_q(s):=\Var_{\nu_{q,s}}D_q(s)\longrightarrow0,
\]
by Theorem~\ref{ii:thm:relative} and
\eqref{ii:lim:eq:Dvariance}. The real normalized fibre volume is
$V=(4\pi)^{n-m}(-K_F)^{n-m}/(n-m)!$.

Fix $q$, set $h_0=\omega_F$, and write
$h_q(s)=h_0+\ddc\varphi(s)$ and $V_0=\int_Fh_0^{n-m}$.
By \eqref{ii:lim:eq:forced-flow},
$\dot\varphi=f_q+D_q+c(s)$.
The Aubin--Mabuchi energy and its first variation are
\cite[Section~5.1]{DarvasRubinstein}
\[
 E(\varphi)=\frac1{(n-m+1)V_0}
       \sum_{j=0}^{n-m}\int_F\varphi\,h_q^j\wedge h_0^{n-m-j},
 \qquad \dot E=\int_F\dot\varphi\,d\nu_{q,s}.
\]
Define
\[
 \Ding(\varphi)=-E(\varphi)
       -\log\left(V_0^{-1}\int_Fe^{-\varphi}h_0^{n-m}\right).
\]
Since $E(\varphi+c)=E(\varphi)+c$, the functional is unchanged
by adding constants. It is bounded below because $F$ admits a
K\"ahler--Einstein metric \cite[Theorem~7.1]{DarvasRubinstein}.

\begin{lemma}\label{ii:lem:ding}
Put $A_q(s)=\int_Ff_q^2\,d\nu_{q,s}$ and
$B_q(s)=\Var_{\nu_{q,s}}D_q$. If $|f_q|\le K$, then
\begin{equation}\label{ii:eq:ding-ineq}
 \frac{d}{ds}\Ding(\varphi(s))
 \le-\frac12e^{-K}A_q(s)+\frac12e^{3K}B_q(s).
\end{equation}
There is a sequence $s_i\to\infty$ with $A_q(s_i)\to0$.
\end{lemma}

\begin{proof}
The exponentially normalized intrinsic Ricci potential is
\[
 f_q=\log\frac{h_q^{n-m}}{h_0^{n-m}}+\varphi
       +\log\left(V_0^{-1}\int_Fe^{-\varphi}h_0^{n-m}\right).
\]
Thus $e^{-f_q}d\nu_{q,s}=e^{-\varphi}h_0^{n-m}/\int_Fe^{-\varphi}h_0^{n-m}$.
Put $\overline D_q=\int_FD_q\,d\nu_{q,s}$.
Using \eqref{ii:lim:eq:forced-flow} and $\int(e^{-f_q}-1)d\nu_{q,s}=0$,
\begin{align*}
 \dot\Ding
 &=-\int_F\dot\varphi\,d\nu_{q,s}
      +\int_F\dot\varphi e^{-f_q}\,d\nu_{q,s}\\
 &=\int_F(f_q+D_q)(e^{-f_q}-1)\,d\nu_{q,s}\\
 &=-\int_F f_q(1-e^{-f_q})\,d\nu_{q,s}
   +\int_F(D_q-\overline D_q)(e^{-f_q}-1)\,d\nu_{q,s}.
\end{align*}
For $|x|\le K$,
\[
 x(1-e^{-x})=x^2\int_0^1e^{-tx}\,dt\ge e^{-K}x^2,
 \qquad |e^{-x}-1|\le e^K|x|.
\]
Writing $A=A_q(s)$ and $B=B_q(s)$, Cauchy--Schwarz gives
\[
 \left|\int_F(D_q-\overline D_q)(e^{-f_q}-1)\,d\nu_{q,s}\right|
 \le e^K\sqrt{AB}.
\]
Young's inequality gives
\[
 e^K\sqrt{AB}
 =(e^{-K/2}\sqrt A)(e^{3K/2}\sqrt B)
 \le\tfrac12e^{-K}A+\tfrac12e^{3K}B.
\]
Combining these estimates with the negative term
$-e^{-K}A$ proves \eqref{ii:eq:ding-ineq}. If $\Ding\ge-L$,
integration gives
\[
 \frac{e^{-K}}2\int_{s_0}^{S}A_q(s)\,ds
 \le\Ding(\varphi(s_0))+L
        +\frac{e^{3K}}2\int_{s_0}^{S}B_q(s)\,ds.
\]
Fix $\epsilon>0$ and choose
$S_\epsilon$ so that $B_q(s)\le\epsilon$ for $s\ge S_\epsilon$,
using \eqref{ii:lim:eq:Dvariance}. Then, for $S>S_\epsilon$,
\[
 \frac1{S-s_0}\int_{s_0}^SB_q(s)\,ds
 \le\frac{\int_{s_0}^{S_\epsilon}B_q(s)\,ds}{S-s_0}
       +\epsilon\frac{S-S_\epsilon}{S-s_0}.
\]
First let $S\to\infty$ and then $\epsilon\to0$. Dividing the
integrated Ding inequality by $S-s_0$ therefore yields
\[
 \frac1{S-s_0}\int_{s_0}^SA_q(s)\,ds\longrightarrow0.
\]
In particular, for every integer $j$ there is $s_j\ge j+s_0$
with $A_q(s_j)<1/j$. Otherwise $A_q\ge1/j$ on that whole tail,
contradicting the vanishing average. These are the required times.
\end{proof}

\begin{proposition}\label{ii:prop:identify-KE}
Every fixed or moving Type-I-scale limit under the
K\"ahler--Einstein fibre hypothesis has smooth time-$-1$ slice
holomorphically isometric to
$(\C^m\times F,g_{\mathrm{Euc}}+g_F)$.
The convergence on regular spacetime includes the metrics and
complex structures; after a constant invertible complex-linear
change of base coordinates, the magnified base maps converge to
the Euclidean projection.
\end{proposition}
\begin{proof}
\emph{Step 1: a K\"ahler--Einstein tangent at each fixed fibre.}
Fix $q$ and take the times from Lemma~\ref{ii:lem:ding}.
The compactness, convergence and volume statements of
Proposition~\ref{ii:lim:prop:soliton-static} give a residual $S$
and regular fibre graphs along a subsequence. By
\eqref{ii:lim:eq:residual-potential}, the pulled-back intrinsic
potentials converge smoothly to the normalized soliton potential:
$f_q(s_i)\to u_S$. For $U\Subset\Reg S$, with fibre graph $U_i$,
\[
 \frac1V\int_Uu_S^2\,dV_S
 =\lim_{i\to\infty}\int_{U_i}f_q(s_i)^2\,d\nu_{q,s_i}
 \le\lim_{i\to\infty}A_q(s_i)=0.
\]
Thus $u_S=0$. Passing
$\Ric(h_q(s_i))+\ddc f_q(s_i)=h_q(s_i)$ to the regular limit gives
$\Ric(\omega_S)=\omega_S$.
The tangent-entropy identity
\cite[Theorem~2.37]{BamlerStructure}, together with
\eqref{ii:lim:eq:entropy-volume}, shows that the limiting entropy
at every fixed-fibre point is
\begin{equation}\label{ii:eq:Ntop}
 \Nash_*:=\log\frac{V}{(4\pi)^{n-m}}-(n-m)
          =\log\frac{(-K_F)^{n-m}}{(n-m)!}-(n-m).
\end{equation}

\emph{Step 2: every moving residual is K\"ahler--Einstein.}
Proposition~\ref{ii:lim:prop:dini}, applied with $\Theta=\Nash_*$,
and the splitting argument in Subsection~\ref{ii:lim:sec:moving}
give split shrinking limits for arbitrary moving sequences.
Their residuals are compact and have volume $V$ by
Proposition~\ref{ii:lim:prop:soliton-static}. Put
$d\nu_S=V^{-1}dV_S$ and $d\sigma_S=e^{-u_S}d\nu_S$.
The entropy identity \eqref{ii:lim:eq:entropy-volume} becomes
\[
 \Nash_*=\log\frac{V}{(4\pi)^{n-m}}-(n-m)
               -\Ent(\sigma_S\mid\nu_S)
          =\Nash_*-\Ent(\sigma_S\mid\nu_S).
\]
Hence the relative entropy is zero, so $\sigma_S=\nu_S$ and
$u_S=0$. Every such residual is therefore K\"ahler--Einstein.

\emph{Step 3: identification with the original fibre.}
Lemma~\ref{ii:lim:lem:qfano} makes $S$ a $\Q$-Fano variety with
full anticanonical mass. Both $S$ and $F$ are K-polystable by
\cite[Theorem~1.1]{Berman}. Compare the $\Q$-Gorenstein family
in Proposition~\ref{ii:lim:prop:flat} with $F\times C$.
They are isomorphic away from $0$ and satisfy
\cite[Definition~2.1]{BlumXu}. K-polystable separatedness
\cite[Theorem~1.1(2)]{BlumXu} gives $S\cong F$.
Equality of the metric and complex-analytic regular loci
\cite[Theorem~B(iii)]{HZ} makes the limiting metric smooth
everywhere. Bando--Mabuchi uniqueness \cite{BandoMabuchi}
identifies it with $\omega_F$, up to a holomorphic automorphism.
The regular spacetime and base-map convergence are the statements
of Subsection~\ref{ii:lim:sec:moving}.
\end{proof}

\subsection{Full curvature and intrinsic fibre diameter}\label{ii:subsec:transfer}

It remains to transfer the smooth limiting geometry to the whole
original fibre. A priori, regular convergence only describes the
regions captured by the convergence maps. When the limiting fibre
is smooth and compact, the implicit function theorem gives a
compact graph inside the exact original fibre. This graph is both
open and closed, so connectedness makes it the whole fibre. The
curvature and diameter estimates then follow by applying this
observation to arbitrary moving sequences.

\begin{lemma}\label{ii:lem:whole-fibre}
Let $\Phi:X\to Y$ satisfy the hypotheses of
Theorem~\ref{ii:thm:main}. Suppose that a moving Type-I-scale
limit is the smooth product
$\C^m\times S_*$, where $S_*$ is compact and connected, and that
the normalized base maps converge to its projection. Then one
smooth convergence region contains the entire original central
fibre for all large indices. Its ambient curvature is uniformly
bounded there in the rescaled metrics, and the induced normalized
fibre metrics converge smoothly to the metric of $S_*$ after
diffeomorphisms.
\end{lemma}
\begin{proof}
At time $-1$, choose a convergence embedding $\Psi_i$ defined on a
neighborhood of $\overline B(0,2a)\times S_*$, for a fixed $a>0$,
using \cite[Theorems~2.5 and~2.14]{BamlerStructure}.
After a constant linear change of base variables,
\[
 \mathcal B_i\circ\Psi_i\longrightarrow
                   \operatorname{pr}_{\C^m}
 \quad\hbox{smoothly}.
\]
Put $b_i(z,y)=\mathcal B_i(\Psi_i(z,y))$. On the fixed compact
product, taking derivatives in the underlying real coordinates,
smooth convergence gives
\[
 \sup_y|b_i(0,y)|\longrightarrow0,\qquad
 \sup_{(z,y)}\|D_zb_i(z,y)-I\|\longrightarrow0.
\]
For large $i$, we have $\sup_y|b_i(0,y)|\le a/2$ and
$\sup\|I-D_zb_i\|\le1/2$. Thus, for $|z|\le a$,
\begin{align*}
 |z-b_i(z,y)|
 &\le |b_i(0,y)|+
       \int_0^1\|I-D_zb_i(tz,y)\|\,|z|\,dt\\
 &\le a/2+|z|/2\le a.
\end{align*}
The map $z\mapsto z-b_i(z,y)$ is therefore a contraction of the
closed ball into itself. It has a unique fixed point $u_i(y)$ in that
ball. The parameter-dependent implicit function theorem gives
smooth dependence on $y$ and, by differentiating
$b_i(u_i(y),y)=0$, smooth convergence to zero. Thus
\[
 u_i\to0\text{ in }C^\infty,\qquad
 \mathcal B_i\bigl(\Psi_i(u_i(y),y)\bigr)=0.
\]
For example, the first derivative is
\[
 du_i=-(D_zb_i)^{-1}D_yb_i
             \quad\hbox{along }(u_i(y),y),
\]
and higher derivatives follow inductively because
$(D_zb_i)^{-1}$ is uniformly bounded. Uniqueness makes
the constructions in different coordinate charts of $S_*$ agree.

Let $j_i(y)=(u_i(y),y)$ and $\iota_i=\Psi_i\circ j_i$.
Its image $\Sigma_i$ is a compact embedded submanifold of the exact
fibre $F_{q_i}$. Indeed, $\mathcal B_i=0$ means
$w\circ\Phi=w(q_i)$ in this product chart, hence $\Phi=q_i$.
The differential of $\mathcal B_i\circ\Psi_i$ has real rank $2m$
along the graph. Its zero set therefore has real dimension $2(n-m)$,
the same as $F_{q_i}$, and the local zero-set description shows
that $\Sigma_i$ is open in $F_{q_i}$. Compactness makes it closed.
As $F_{q_i}$ is connected and $\Sigma_i$ is nonempty,
\[
                         \Sigma_i=F_{q_i}.
\]
Thus $\iota_i:S_*\to F_{q_i}$ is a diffeomorphism. Writing
$\widehat g_i=\Psi_i^*g_i(-1)$, one has the exact pullback identity
\[
 \iota_i^*(g_i(-1)|_{F_{q_i}})=j_i^*\widehat g_i.
\]
Since $j_i\to j_\infty$, $j_\infty(y)=(0,y)$, and
$\widehat g_i\to g_{\mathrm{Euc}}+g_{S_*}$ smoothly on a fixed
larger neighborhood, the right-hand side converges smoothly to
$g_{S_*}$. The same ambient convergence gives a uniform curvature
bound on the graph, hence on the whole original fibre. No assertion
that the graph diffeomorphisms are biholomorphic is needed.
\end{proof}

\begin{proposition}\label{ii:prop:diameter-limit}
Under the hypotheses of Theorem~\ref{ii:thm:einstein}, put
$d_F=\diam(F,g_F)$. Then
\begin{equation}\label{ii:eq:diameter-limit}
 \sup_{q\in Y}\left|
 \frac{\diam_{(F_q,g(t)|_{F_q})}F_q}{\sqrt{T-t}}-d_F
 \right|\longrightarrow0\qquad(t\nearrow T).
\end{equation}
\end{proposition}

\begin{proof}[Proof of Theorem~\ref{ii:thm:einstein} and
Proposition~\ref{ii:prop:diameter-limit}]
Let $q_i\in Y$, $t_i\nearrow T$, $\tau_i=T-t_i$, and
$g_i(a)=\tau_i^{-1}g(T+\tau_i a)$.
Proposition~\ref{ii:prop:identify-KE} and
Lemma~\ref{ii:lem:whole-fibre} give, after a subsequence,
diffeomorphisms $\iota_i:F\to F_{q_i}$ such that
\[
 \iota_i^*\bigl(\tau_i^{-1}g(t_i)|_{F_{q_i}}\bigr)
       \longrightarrow g_F\quad\hbox{smoothly},\qquad
 \sup_{F_{q_i}}|\Rm(g_i(-1))|\le C.
\]
The covariant curvature tensor scales by $c$ under $g\mapsto cg$,
whereas its squared norm uses four inverse metrics. Hence
\[
 |\Rm(cg)|_{cg}^2=c^2c^{-4}|\Rm(g)|_g^2,\qquad
 \tau_i\sup_{F_{q_i}}|\Rm(g(t_i))|\le C.
\]
Failure of \eqref{ii:eq:typeI} would yield a sequence for which
the last quantity diverges, a contradiction.

For some $\varepsilon_i\searrow0$, the pulled-back fibre metrics
lie between $(1-\varepsilon_i)g_F$ and $(1+\varepsilon_i)g_F$.
Taking lengths, infima over curves and suprema over endpoints gives
\[
 \sqrt{1-\varepsilon_i}\,d_F
 \le\tau_i^{-1/2}\diam_{(F_{q_i},g(t_i)|_{F_{q_i}})}F_{q_i}
 \le\sqrt{1+\varepsilon_i}\,d_F.
\]
A sequence violating uniform convergence would contradict these
inequalities, proving \eqref{ii:eq:diameter-limit}. Since $n-m>0$,
$d_F>0$, so the two-sided bound \eqref{ii:eq:intrinsic-diameter}
follows near $T$. On a remaining compact time interval, smooth
metric equivalence and a finite trivializing cover give uniform
curvature and positive lower and finite upper intrinsic diameter
bounds. Enlarging $C$ proves assertions \textup{(1)} and \textup{(2)}
of Theorem~\ref{ii:thm:einstein}.

The residual potential vanishes by the proof of
Proposition~\ref{ii:prop:identify-KE}. Thus every limiting flow is
\eqref{ii:eq:einstein-model}; its normalization is
\[
 \partial_a\bigl(g_{\mathrm{Euc}}+(-a)g_F\bigr)
 =-g_F=-2\Ric\bigl(g_{\mathrm{Euc}}+(-a)g_F\bigr).
\]
The regular spacetime convergence is smooth on the entire limiting
product. If central points $x_i\in F_{q_i}$ are marked, write
$x_i=\iota_i(y_i)$ and take a subsequence $y_i\to y_\infty$ in
compact $F$. Their preimages in the product are
$p_i=(u_i(y_i),y_i)\to(0,y_\infty)$. A diagonal exhaustion of
compact negative-time intervals, followed by precomposition with
spatial diffeomorphisms tending smoothly to the identity and
sending $(0,y_\infty)$ to $p_i$, gives pointed smooth
Cheeger--Gromov convergence. The complex structures converge under
the same maps, since the limiting product is the smooth K\"ahler limit
identified in Proposition~\ref{ii:prop:identify-KE}.
The base-map conclusion follows from
Proposition~\ref{ii:prop:identify-KE}, proving
Theorem~\ref{ii:thm:einstein}\textup{(3)}.
\end{proof}

For the compact-local assertions in Remark~\ref{ii:rem:local},
apply the same argument to sequences above a fixed compact subset
of the regular product locus. A convergent base subsequence lies
in one larger product chart. All analytic and limiting inputs are
uniform there, and the graph construction stays inside this chart.
The resulting estimates have constants depending on the compact
subset.


\begin{thebibliography}{99}

\bibitem{Stacks}
The Stacks Project Authors, \emph{The Stacks Project}.
Perfect direct images and base change: Tag~0A1G,
\url{https://stacks.math.columbia.edu/tag/0A1G}.
Normality under flat morphisms: Tag~0C22,
\url{https://stacks.math.columbia.edu/tag/0C22}.
Weil divisors on normal schemes: Tag~0EBM,
\url{https://stacks.math.columbia.edu/tag/0EBM}.

\bibitem{BamlerCompactness}
R. H. Bamler,
\emph{Compactness theory of the space of super Ricci flows},
Invent. Math. \textbf{233} (2023), 1121--1277.
Preprint: \href{https://arxiv.org/abs/2008.09298v2}{arXiv:2008.09298v2}.

\bibitem{BamlerHeat}
R. H. Bamler, \emph{Entropy and heat kernel bounds on a Ricci flow background},
\href{https://arxiv.org/abs/2008.07093v3}{arXiv:2008.07093v3} (2021).

\bibitem{BamlerStructure}
R. H. Bamler,
\emph{Structure theory of non-collapsed limits of Ricci flows},
\href{https://arxiv.org/abs/2009.03243v2}{arXiv:2009.03243v2}.

\bibitem{BandoMabuchi}
S. Bando and T. Mabuchi,
\href{https://doi.org/10.2969/aspm/01010011}{\emph{Uniqueness of Einstein K\"ahler metrics modulo connected group
actions}}, in \emph{Algebraic Geometry, Sendai, 1985}, Adv. Stud.
Pure Math. \textbf{10} (1987), 11--40.

\bibitem{BedfordTaylor}
E. Bedford and B. A. Taylor,
\href{https://doi.org/10.1007/BF02392348}{\emph{A new capacity for plurisubharmonic functions}},
Acta Math. \textbf{149} (1982), 1--40.

\bibitem{Berman}
R. J. Berman,
\emph{K-polystability of $\Q$-Fano varieties admitting
K\"ahler--Einstein metrics}, Invent. Math. \textbf{203} (2016),
973--1025; \href{https://arxiv.org/abs/1205.6214v3}{arXiv:1205.6214v3}.

\bibitem{BlumXu}
H. Blum and C. Xu,
\emph{Uniqueness of K-polystable degenerations of Fano varieties},
Ann. of Math. (2) \textbf{190} (2019), 609--656;
\href{https://arxiv.org/abs/1812.03538v2}{arXiv:1812.03538v2}.

\bibitem{CCHZ}
C. Cifarelli, R. J. Conlon, M. Hallgren and J. Zhang,
\emph{Finite time singularities of the Ricci flow on compact
K\"ahler surfaces are of Type I},
\href{https://arxiv.org/abs/2609.16733}{arXiv:2609.16733} (2026).

\bibitem{CDSIII}
X. Chen, S. Donaldson and S. Sun,
\emph{K\"ahler--Einstein metrics on Fano manifolds. III: Limits as
cone angle approaches $2\pi$ and completion of the main proof},
J. Amer. Math. Soc. \textbf{28} (2015), 235--278;
\href{https://arxiv.org/abs/1302.0282}{arXiv:1302.0282}.

\bibitem{Chow}
W. L. Chow,
\href{https://doi.org/10.2307/2372375}{\emph{On compact complex analytic varieties}},
Amer. J. Math. \textbf{71} (1949), no.~4, 893--914.

\bibitem{CollinsSzekelyhidi}
T. C. Collins and G. Sz\'ekelyhidi, \emph{The twisted K\"ahler--Ricci flow},
J. Reine Angew. Math. \textbf{716} (2016), 179--205.
Preprint: \href{https://arxiv.org/abs/1207.5441v2}{arXiv:1207.5441v2}.

\bibitem{DarvasRubinstein}
T. Darvas and Y. A. Rubinstein,
\emph{Tian's properness conjectures and Finsler geometry of the
space of K\"ahler metrics}, J. Amer. Math. Soc. \textbf{30} (2017),
347--387; \href{https://arxiv.org/abs/1506.07129v2}{arXiv:1506.07129v2}.

\bibitem{Demailly}
J.-P. Demailly, \emph{Complex Analytic and Differential Geometry},
online book, version of June 21, 2012.
\href{https://www-fourier.univ-grenoble-alpes.fr/~demailly/manuscripts/agbook.pdf}{Institut Fourier manuscript}.

\bibitem{DemaillyL2}
J.-P. Demailly,
\emph{$L^2$ estimates for the $\bar\partial$-operator on complex
manifolds}, author lecture notes.
\href{https://www-fourier.univ-grenoble-alpes.fr/~demailly/manuscripts/estimations_l2.pdf}{Institut Fourier lecture notes}.

\bibitem{FuZhang}
X. Fu and S. Zhang, \href{https://doi.org/10.1007/s00209-017-1881-4}{\emph{The K\"ahler--Ricci flow on Fano bundles}},
Math. Z. \textbf{286} (2017), no. 3--4, 1605--1626.
Preprint: \href{https://arxiv.org/abs/1610.03366v2}{arXiv:1610.03366v2}.

\bibitem{GilbargTrudinger}
D. Gilbarg and N. S. Trudinger, \href{https://doi.org/10.1007/978-3-642-61798-0}{\emph{Elliptic Partial Differential Equations of Second Order}},
second edition, Classics in Mathematics, Springer-Verlag, Berlin, 2001.

\bibitem{HZ}
M. Hallgren and J. Zhang,
\emph{Singular K\"ahler--Ricci shrinkers and polarized Fano fibrations},
\href{https://arxiv.org/abs/2605.25213v2}{arXiv:2605.25213v2}, 31 August 2026.


\bibitem{Splitting}
W. Jian and J. Song,
\emph{Finite-time singularities of the K\"ahler--Ricci flow on Fano
bundles}, \href{https://arxiv.org/abs/2609.02878v1}{arXiv:2609.02878v1} (2026).

\bibitem{JSfb3}
W. Jian and J. Song,
\emph{Finite-time singularities of the K\"ahler--Ricci flow on Fano
bundles III}, note primarily generated by AI, available at \href{https://sites.google.com/site/jiansongrutgers/home/research}{https://sites.google.com/site/jiansongrutgers/home/research}


\bibitem{JST}
W. Jian, J. Song and G. Tian,
\href{https://doi.org/10.1007/s00222-026-01449-x}{\emph{Finite time singularities of the K\"ahler--Ricci flow}},
Invent. Math. (2026), published online 15 September 2026.
Preprint: \href{https://arxiv.org/abs/2310.07945v2}{arXiv:2310.07945v2}
(the numbering used here).



\bibitem{PSSWPositive}
D. H. Phong, J. Song, J. Sturm and B. Weinkove,
\href{https://doi.org/10.1007/s00222-008-0134-x}{\emph{The K\"ahler--Ricci flow with positive bisectional curvature}},
Invent. Math. \textbf{173} (2008), no.~3, 651--665.

\bibitem{SSW}
J. Song, G. Sz\'ekelyhidi and B. Weinkove, \href{https://doi.org/10.1093/imrn/rnr265}{\emph{The K\"ahler--Ricci flow on projective bundles}},
Int. Math. Res. Not. IMRN \textbf{2013} (2013), no. 2, 243--257.
Preprint: \href{https://arxiv.org/abs/1107.2144}{arXiv:1107.2144}.

\bibitem{SongTianSurfaces}
J. Song and G. Tian,
\href{https://doi.org/10.1007/s00222-007-0076-8}{\emph{The K\"ahler--Ricci flow on surfaces of positive Kodaira dimension}},
Invent. Math. \textbf{170} (2007), no.~3, 609--653.

\bibitem{SongTianMeasures}
J. Song and G. Tian,
\href{https://doi.org/10.1090/S0894-0347-2011-00717-0}{\emph{Canonical measures and K\"ahler--Ricci flow}},
J. Amer. Math. Soc. \textbf{25} (2012), no.~2, 303--353.

\bibitem{SongTianSingularities}
J. Song and G. Tian,
\href{https://doi.org/10.1007/s00222-016-0674-4}{\emph{The K\"ahler--Ricci flow through singularities}},
Invent. Math. \textbf{207} (2017), no.~2, 519--595.

\bibitem{TianKE}
G. Tian,
\href{https://doi.org/10.1002/cpa.21578}{\emph{K-stability and K\"ahler--Einstein metrics}},
Comm. Pure Appl. Math. \textbf{68} (2015), 1085--1156;
\href{https://arxiv.org/abs/1211.4669}{arXiv:1211.4669}.
\href{https://doi.org/10.1002/cpa.21612}{Corrigendum}, Comm. Pure Appl. Math. \textbf{68} (2015),
no.~11, 2082--2083.

\bibitem{TianZhangActa}
G. Tian and Z. Zhang,
\href{https://doi.org/10.1007/s11511-016-0137-1}{\emph{Regularity of K\"ahler--Ricci flows on Fano manifolds}},
Acta Math. \textbf{216} (2016), no.~1, 127--176.

\bibitem{TianZhuFlow}
G. Tian and X. Zhu,
\href{https://doi.org/10.1090/S0894-0347-06-00552-2}{\emph{Convergence of K\"ahler--Ricci flow}},
J. Amer. Math. Soc. \textbf{20} (2007), no.~3, 675--699.

\bibitem{TianZhuFlowII}
G. Tian and X. Zhu,
\href{https://doi.org/10.1515/crelle.2012.021}{\emph{Convergence of the K\"ahler--Ricci flow on Fano manifolds}},
J. Reine Angew. Math. \textbf{678} (2013), 223--245.

\bibitem{XuZhang}
T. Xu and Z. Zhang,
\emph{Finite time singularities of collapsing K\"ahler Ricci flow on ruled surfaces},
\href{https://arxiv.org/abs/2609.01442}{arXiv:2609.01442} (2026).

\bibitem{XuZhangSurfaces}
T. Xu and Z. Zhang,
\emph{Collapsed finite time singularities of the K\"ahler--Ricci
flow on complex surfaces are of Type I},
\href{https://arxiv.org/abs/2609.18834}{arXiv:2609.18834} (2026).

\bibitem{Ye}
R. Ye,
\emph{The logarithmic Sobolev and Sobolev inequalities along the Ricci flow},
Comm. Math. Stat. \textbf{3} (2015), no.~1, 1--36.
Preprint: \href{https://arxiv.org/abs/0707.2424v4}{arXiv:0707.2424v4},
\emph{The logarithmic Sobolev inequality along the Ricci flow}.

\bibitem{ZhouZhu}
X. Zhou and L. Zhu,
\emph{Optimal $L^2$ extension of sections from subvarieties in weakly
pseudoconvex manifolds}, Pacific J. Math. \textbf{309} (2020),
475--510; \href{https://arxiv.org/abs/1909.08820v1}{arXiv:1909.08820v1}.

\end{thebibliography}
\end{document}